\documentclass[12pt]{amsart}
\font\bbbld=msbm10 scaled\magstephalf

\newcommand{\bi}{\bar{i}}
\newcommand{\bj}{\bar{j}}
\newcommand{\bk}{\bar{k}}
\newcommand{\bl}{\bar{l}}
\newcommand{\bm}{\bar{m}}
\newcommand{\bn}{\bar{n}}
\newcommand{\bp}{\bar{p}}
\newcommand{\bq}{\bar{q}}
\newcommand{\br}{\bar{r}}
\newcommand{\bs}{\bar{s}}
\newcommand{\bw}{\bar{w}}
\newcommand{\bz}{\bar{z}}

\newcommand{\bM}{\bar{M}}

\newcommand{\balpha}{\bar{\alpha}}

\newcommand{\bpartial}{\bar{\partial}}

\def \p{\partial}
\def \f{\frac}

\def \o{\omega}
\def \d{\delta}

\def \tg{\tilde{g}}

\newcommand{\fg}{\mathfrak{g}}
\newcommand{\fIm}{\mathfrak{Im}}
\newcommand{\fRe}{\mathfrak{Re}}

\newcommand{\bfC}{\hbox{\bbbld C}}

\newcommand{\bfQ}{\hbox{\bbbld Q}}
\newcommand{\bfR}{\hbox{\bbbld R}}
\newcommand{\bfS}{\hbox{\bbbld S}}

\newcommand{\cA}{\mathcal{A}}

\newcommand{\cH}{\mathcal{H}}

\newcommand{\cT}{\mathcal{T}}

\newcommand{\tr}{\mbox{tr}}

\newcommand{\ol}{\overline}
\newcommand{\ul}{\underline}

\newtheorem{theorem}{Theorem}[section]
\newtheorem{lemma}[theorem]{Lemma}
\newtheorem{proposition}[theorem]{Proposition}

\newtheorem{conjecture}[theorem]{Conjecture}
 \theoremstyle{definition}

\theoremstyle{remark}
\newtheorem{remark}[theorem]{Remark}

\numberwithin{equation}{section}

\begin{document}
\setlength{\baselineskip}{1.2\baselineskip}

\title[Complex Monge-Amp\`ere Equations]
{Complex Monge-Amp\`ere Equations and \\ Totally Real Submanifolds}
\author{Bo Guan}
\address{Department of Mathematics, Ohio State University,
         Columbus, OH 43210}
\email{guan@math.osu.edu}
\thanks{Research of the first author was supported in part by NSF grants.}
\author{Qun Li}
\address{Department of Mathematics, Ohio State University,
         Columbus, OH 43210. 
{\em Current address}: 
Department of Mathematics, Wright State University,
         Dayton, OH 45435}
\email{qun.li@wright.edu}

\begin{abstract}
We study the Dirichlet problem for complex Monge-Amp\`ere equations in 
Hermitian manifolds with general (non-pseudoconvex) boundary.
Our main result (Theorem~\ref{gblq-th20})
 extends the classical theorem of 
Caffarelli, Kohn, Nirenberg and Spruck~\cite{CKNS} in $\bfC^n$. 
We also consider the equation on compact manifolds without boundary,
attempting to generalize Yau's theorems~\cite{Yau78} in the K\"ahler case. 
As applications of the main result we study some connections between the
homogeneous complex Monge-Amp\`ere ({\em HCMA}) equation and totally real 
submanifolds, and a special Dirichlet problem for the HCMA equation
 related to Donaldson's conjecture~\cite{Donaldson99} on geodesics in
the space of K\"ahler metrics.

{\bf Mathematical Subject Classification (2000). } 58J05, 58J32, 32W20, 35J25,
 53C55.
\end{abstract}

\maketitle

\bigskip

\section{Introduction}
\label{gblq-I}
\setcounter{equation}{0}
\medskip

There are two primary purposes in this paper which are closely related.
One is to study the Dirichlet problem for complex Monge-Amp\`ere
type equations in Hermitian manifolds. The other is to characterize
totally real submanifolds by solutions of the homogeneous Monge-Amp\`ere
equation using results from the first part. The latter is also one of the
original motivations to our study of Monge-Amp\`ere type equations on
general Hermitian manifolds.

Let $(M^n, \omega)$ be a compact Hermitian manifold of (complex) dimension
$n \geq 2$ with smooth boundary $\partial M$,  and $\bM = M \cup \partial M$.
Let $\chi$ be a smooth $(1,1)$ form on $M$ and
$\psi \in C^{\infty} (M \times \bfR)$. 
We consider the Dirichlet
problem for the complex Monge-Amp\`ere equation
\begin{equation}
\label{gblq-I10}
   \Big(\chi + \frac{\sqrt{-1}}{2} \partial \bpartial u\Big)^n
  = \psi (z, u) \omega^n \;\; \mbox{in $\bM$}.
\end{equation}
Given $\varphi \in C^{\infty} (\partial M)$, we seek solutions
of equation~\eqref{gblq-I10}
satisfying the boundary condition
\begin{equation}
\label{gblq-I20}
  u = \varphi \;\; \mbox{on $\partial M$}.
 \end{equation}

We require
\begin{equation}
\label{gblq-I30}
\chi_u :=\chi + \frac{\sqrt{-1}}{2} \partial \bpartial u > 0
 \end{equation}
so that equation~\eqref{gblq-I10} is elliptic; we call such functions
{\em admissible}. Accordingly we shall assume $\psi > 0$.
Equation~\eqref{gblq-I10} becomes degenerate for $\psi \geq 0$; when
$\psi \equiv 0$ it is usually referred as the
homogeneous complex Monge-Amp\`ere ({\em HCMA}) equation.
Set
\begin{equation}
\label{gblq-I40}
 \cH_{\chi} = \{\phi \in C^2 (\bM): \chi_{\phi} > 0\}, \;\;
 \bar{\cH}_{\chi} = \{\phi \in C^0 (\bM): \chi_{\phi} \geq 0\}.
\end{equation}

Two canonical cases that are very important in
complex geometry and analysis correspond to $\chi = \omega$
and $\chi = 0$.  For $u \in \cH_{\omega}$, as in the K\"ahler case,
$\omega_u$ is a Hermitian form on $M$ and equation~\eqref{gblq-I10}
describes one of its Ricci forms.
We call $\cH_{\omega}$ the {\em space of Hermitian metrics}.
For $\chi = 0$, functions in $\cH_{\chi}$ are strictly plurisubharmonic,
while those in $\bar{\cH}_{\chi}$ plurisubharmonic.

The classical solvability of the Dirichlet problem was established by
Caffarelli, Kohn, Nirenberg and Spruck~\cite{CKNS} for strongly pseudoconvex
domains in $\bfC^n$.
Their results were extended to strongly pseudoconvex Hermitian manifolds
by Cherrier and Hanani~\cite{CH98}, \cite{CH99}
(for $\chi =0, \; \omega, \; - u\omega$ in \eqref{gblq-I10}),
and to general domains in
$\bfC^n$ by the first author~\cite{Guan98b} under the assumption of
existence of a subsolution. This latter extension and its techniques
have found useful applications in some important work;
see, e.g., P.-F. Guan's proof~\cite{GuanPF02}, \cite{GuanPF08}
of Chern-Levine-Nirenberg conjecture~\cite{CLN69} and the papers
of Chen~\cite{Chen00}, Blocki~\cite{Blocki}, and Phong and Sturm~\cite{PS}
 on the Donaldson conjectures~\cite{Donaldson99}.
Our first purpose in this paper is to treat the Dirichlet problem in
general (non-pseudoconvex) Hermitian manifolds.

\begin{theorem}
\label{gblq-th20}
Suppose that $\psi > 0$ and that there exists a
subsolution $\ul{u} \in \cH_{\chi} \cap C^4 (\bM)$
of \eqref{gblq-I10}-\eqref{gblq-I20}:
\begin{equation}
\label{gblq-I10'}
\left\{ \begin{aligned}
  & (\chi_{\ul{u}})^n \geq  \psi (z, \ul{u}) \omega^n \;\; \mbox{in $\bM$}
\\
& \ul{u} = \varphi  \;\; \mbox{on $\partial M$}
\end{aligned} \right.
 \end{equation}
The Dirichlet problem~\eqref{gblq-I10}-\eqref{gblq-I20} then admits a solution
$u \in \cH_{\chi} \cap C^{\infty} (\bM)$. 
\end{theorem}

When $\chi > 0$, which is not assumed in Theorem~\ref{gblq-th20}, the
conditions on $\ul{u}$ can be weakened: it is enough to assume
$\ul{u} \in \bar{\cH}_{\chi}$; $\ul{u} \in C^2$ and $\chi_{\ul{u}} > 0$
in a neighborhood of $\partial M$, and satisfies \eqref{gblq-I10'} in the
viscosity sense. This shall be convenient in applications.

The Monge-Amp\`ere equation is one of the most important partial
differential equations in complex geometry and analysis.
In the framework of K\"ahler geometry,
it goes back at least to the Calabi conjecture~\cite{Calabi56}
which asserts that any element in the first Chern class of a compact K\"ahler
manifold is the Ricci form of a K\"ahler metric cohomologous to the
underlying metric.
In~\cite{Yau78}, Yau proved fundamental existence theorems
for complex Monge-Amp\`ere equations on compact K\"ahler
manifolds (without boundary) and consequently solved the Calabi
conjecture.
 Yau's work
also shows the existence of K\"ahler-Einstein metrics on K\"ahler
manifolds with nonpositive first Chern class ($c_1 (M) \leq 0$),
proving another Calabi conjecture which was solved
 by Aubin~\cite{Aubin78} independently for $c_1 (M) < 0$.
In a series of work (e.g. \cite{Tian87}, \cite{Tian90}, \cite{Tian97}),  
Tian made important contributions to the
Calabi conjecture when $c_1 (M) > 0$; see also \cite{Aubin98} and 
\cite{Tian00} for more references. 
More recently,  Donaldson~\cite{Donaldson99} made several conjectures
concerning geodesics in the space of K\"ahler metrics which reduce to
questions on special Dirichlet problems for the homogeneous complex
Monge-Amp\`ere ({\em HCMA}) equation; see also
Mabuchi~\cite{Mabuchi87} and Semmes~\cite{Semmes92}. There has been
interesting work in this direction, e.g. by Chen~\cite{Chen00},
Chen and Tian~\cite{CT08}, Phong and Sturm~\cite{PS06},
\cite{PS07}, \cite{PS}, Blocki~\cite{Blocki},
and Berman and Demailly~\cite{BD}.

The HCMA equation
($\psi \equiv 0$ in \eqref{gblq-I10}), which is well defined on general
complex manifolds, also arises in many other interesting geometric problems;
see e.g. \cite{CLN69}, \cite{BT79},
\cite{Stoll80}, {Lempert81}, \cite{Burns82}, \cite{Wong82},
\cite{GuanPF02}, \cite{GuanPF08}.
Because the HCMA equation is degenerate, the optimal regularity
of its solution in general is only $C^{1,1}$; see \cite{BF79},
\cite{GS80}. On the other hand, methods
from complex analysis so far seem to have only been able to produce
smooth or analytic solutions under special circumstances; in order to
treat the equation using techniques of elliptic PDE theory one needs
to introduce a metric on the manifold. In the full generality it seems
most natural to consider Hermitian metrics as every complex manifold
admits such a metric. Moreover, in many problems one needs to consider
manifolds with non-pseudoconvex boundary.
These are some of the major motivations to our study
of the Dirichlet problem on general Hermitian manifolds.
As an application of Theorem~\ref{gblq-th20} we consider some connections
between the HCMA equation and totally real submanifolds.

A submanifold $X$ of a complex manifold $M$ is {\em totally real} if
for any $p \in X$ the tangent space $T_pX$ does not contain any complex
line in $T_p M$,  i.e. $J (T_p X) \cap T_p X=\{0\}$.
In particular, $\mbox{dim} X \leq \frac{1}{2} \mbox{dim}_R M$.
The simplest example is the pair $\bfR^n \subset \bfC^n$, and it is
straightforward to verify that the function $u(z)=|\fIm z|$ satisfies the
HCMA equation $(\partial \bpartial u)^n = 0$ in
$\bfC^n \setminus \bfR^n$.
Another example is the affine hyperquadric
\begin{equation}
\label{Q_n}
\bfQ^n=\{z_1^2+z_2^2+\cdots +z_{n+1}^2=1\}
\end{equation}
which was studied by Patrizio and Wong~\cite{Patrizio-Wong}.
Note that $|z|^2=1 + |\fIm z|^2 \geq 1$ in $\bfQ^n$.
Thus $\bfS^n = \bfQ^n \cap \{|z|=1\}$,
the unit sphere in $\bfR^{n+1}$, is a totally real submanifold of $\bfQ^n$.
It was proved in ~\cite{Patrizio-Wong} that the function
$u=\cosh^{-1}|z|^2$
is plurisubharmonic and satisfies
$(\partial \bpartial u)^n = 0$ in $\bfQ^n \setminus \bfS^n$.
In general any smooth Riemannian manifold is a totally real
submanifold of its cotangent bundle under a canonical complex structure,
and a theorem of Harvey and Wells~\cite{HW73}
says that the minimum set of a $C^2$
plurisubharmonic function is totally real.
In ~\cite{Guillemin-Stenzel91} Guillemin and Stenzel proved that
if $X$ is a compact real-analytic totally real submanifold of dimension
$n$ of a complex manifold $M^n$, then there is a neighborhood $M_1$ of $X$
and a nonnegative plurisubharmonic
solution of $(\partial \bpartial u)^n = 0$ in $M_1 \setminus X$  such that
$X = u^{-1} (\{0\})$ and $u^2$ is smooth and strictly plurisubharmonic
in $M_1$.
For compact symmetric spaces of rank one,
Patrizio and Wong~\cite{Patrizio-Wong} found explicit formulas for
plurisubharmonic solutions of $(\partial \bpartial u)^n = 0$
on corresponding Stein manifolds; see also \cite{Lempert-Szoke}
for related results.
As our second goal in this paper we shall apply Theorem~\ref{gblq-th20}
to prove the following result.

\begin{theorem}
\label{gblq-TR}
Let $X^n$ be a $C^3$ compact totally real submanifold of dimension $n$ in a
 complex manifold $M^n$, and $\chi$ a $(1,1)$-form on $M$.
There exists a tubular neighborhood $M_1$ of $X$ and a (weak) solution
$u \in \bar{\cH}_{\chi} (\bM_1) \cap C^{0,1} (\bM_1)$  of the HCMA equation
\begin{equation}
\label{gblq-I10H}
  (\chi_u)^n = 0 \;\; \mbox{in $\bM_1 \setminus X$}
 \end{equation}
such that $0 \leq u \leq 1$ on $\bM_1$, $X = u^{-1} (\{0\})$ and
$\partial M_1 = u^{-1} (\{1\})$.
\end{theorem}

For the defintion of weak solutions see \cite{BT76}.
The global Lipschitz regularity is the best possible as shown by the
example $u (z) = |\fIm z|$ in $\bfC^n$, which is only $C^{0,1}$ along
$\bfR^n$. It was known to
Lempert and Sz\"oke~\cite{Lempert-Szoke} that
a plurisubharmonic solution to the homogeneous Monge-Amp\`ere equation
on Stein manifolds must have singularities along its minimum set.
So the proof of Theorem~\ref{gblq-TR} involves solving
equation~\eqref{gblq-I10H} with prescribed singularity, which is not always
possible for general elliptic equations. As we shall see in
Section~\ref{gblq-T}, our proof of Theorem~\ref{gblq-TR} makes use of
Theorem~\ref{gblq-th20} in an essential way that it can not be replaced
by the previous results of Cherrier and Hanani~\cite{CH98}, \cite{CH99}.
It is also different from the approach of
Guillemin and Stenzel~\cite{Guillemin-Stenzel91}.

In her thesis \cite{Li} the second author proved the existence of a
bounded plurisubharmonic solution to \eqref{gblq-I10H} for $\chi = 0$. 
It would be interesting to prove $u \in C^{1,1} (\bM_1 \setminus X)$.
This will be treated in \cite{GL2}.

Another interesting problem is to consider extensions of Yau's theorems~\cite{Yau78}
to the Hermitian case, that is, to solve equation~\eqref{gblq-I10} on compact
Hermitian manifolds without boundary. 
A difficult question seems to be how to derive $C^0$ estimates, even for $\chi = \omega$.
Yau's estimate in the K\"ahler case ~\cite{Yau78} makes use
of Moser iteration based on his estimate for $\Delta u$ and Sobolev inequality.
His proof was subsequently
simplified by Kazdan~\cite{Kazdan78} for $n =2$, and by Aubin~\cite{Aubin78}
and Bourguignon independently for arbitrary dimension (see e.g. \cite{Siu87}
and \cite{Tian00}).
Alternative proofs were given by Kolodziej~\cite{Kolodziej98} and
Blocki~\cite{Blocki05} based on the pluripotential theory (\cite{BT82}) and
the $L^2$ stability of the complex Monge-Amp\`ere operator (\cite{CP92}).
All these proofs seem to heavily rely on the closedness or, equivalently,
existence of local potentials of $\omega$ and it is not clear to us whether
any of them can be extended to the Hermitian case.
In this paper we impose the following condition 
\begin{equation}
\label{gblq-I70}
\partial \bpartial \chi = 0, \;
\partial \bpartial \chi^2 = 0
\end{equation}
which is equivalent to $\partial \bpartial \chi = 0$ and
$\partial \chi \wedge \bpartial \chi = 0$.
For $\chi = \omega$, manifolds satisfying \eqref{gblq-I70} were studied
by Fino and Tomassini~\cite{FT}.
We have following extensions of theorems of Yau~\cite{Yau78}.
In Theorems~\ref{gblq-th30} and \ref{gblq-th40} below,
$(M, \omega)$ is a compact Hermitian manifold with
$\partial M = \emptyset$.

\begin{theorem}
\label{gblq-th30}
Assume $\psi_u \geq 0$, $\chi$ satisfy \eqref{gblq-I70},
 and that there exists a function
$\phi \in \cH_{\chi} \cap C^{\infty} (M)$ such that
\begin{equation}
\label{gblq-I80}
\int_M \psi (z, \phi (z)) \omega^n = \int_M \chi^n.
\end{equation}
Then there exists a solution $u \in \cH_{\chi} \cap C^{\infty} (M)$ of
equation~\eqref{gblq-I10}. Moreover the solution is unique, possibly up
to a constant.
\end{theorem}

Consequently, if $\cH_{\chi} \cap C^{\infty} (M) \neq \emptyset$
then for any $\psi \in C^{\infty} (M)$ there is a unique constant $c$
such that equation~\eqref{gblq-I10} has a solution in
$\cH_{\chi} \cap C^{\infty} (M)$ when $\psi$ is replaced by $c \psi$.

For $n=2$, condition~\eqref{gblq-I70} is not needed to derive
$C^0$ bounds; see Remark~\ref{gblq-remark-E30}. For general $n$,
under stronger assumptions on $\psi$ condition
\eqref{gblq-I70} may also be removed.

\begin{theorem}
\label{gblq-th40}
Suppose $\cH_{\chi} \cap C^{\infty} (M) \neq \emptyset$, $\psi_u > 0$ and
\begin{equation}
\label{gblq-I90}
\lim_{u \rightarrow -\infty} \psi (\cdot, u) = 0,  \;\;
 \lim_{u \rightarrow +\infty} \psi (\cdot, u) = \infty.
\end{equation}
Then equation~\eqref{gblq-I10} has a unique solution
in $\cH_{\chi} \cap C^{\infty} (M)$.
\end{theorem}

In \cite{Donaldson06}, Donaldson proposed to generalize 
Yau's theorems in a different direction;  
see \cite{Weinkove07},  \cite{TWY08}, \cite{TW1}, \cite{TW2}
for some recent developments.

The degenerate complex Monge-Amp\`ere equation is very important in
 geometry and analysis. There are many challenging open questions.
Below we formulate some result for a special Dirichlet problem which, in
the K\"ahler case, has been useful in the study of geodesics in the
space of  K\"ahler metrics; see, e.g. \cite{Blocki}, \cite{Chen00},
\cite{PS}.

\begin{theorem}
\label{gblq-th50}
Let $M = N \times S$
where $N$ is a compact Hermitian manifold without
boundary,
and $S$ is a compact Riemann surface with smooth boundary
$\partial S \neq \emptyset$.
Suppose $\psi \geq 0$, $\psi^{\frac{1}{n}} \in  C^2 (\bM \times R)$, and
that there exists a subsolution
$\ul{u} \in \cH_{\chi}$ satisfying \eqref{gblq-I10'}.
Then there exists a weak admissible solution $u \in C^{1,\alpha} (\bM)$,
for all $\alpha \in (0, 1)$ with $\Delta u \in L^{\infty} (M)$ of
the Dirichlet problem \eqref{gblq-I10}-\eqref{gblq-I20}.
\end{theorem}

For applications in geometric problems it would be desirable to
allow $\chi$ to depend on $u$ and its gradient $\nabla u$;
see for instance \cite{FY}. We shall prove Theorem~\ref{gblq-th20} \
for $\chi = \chi (\cdot, u)$ which is 
non-decreasing in $u$, i.e.
$\chi (\cdot, u) - \chi (\cdot, v) \geq 0$ for $u \geq v$.
For the general case
we are able to deal with $\chi = \chi (\cdot, u, \nabla u)$
of the form
\begin{equation}
\label{gblq-I15}
\chi (\cdot, u, \nabla u) = G (|\nabla u|^2) \omega + H (\cdot, u)
   \partial u \wedge \bpartial u
+ \partial u \wedge \balpha (\cdot, u) + \alpha (\cdot, u) \wedge \bpartial u
+ \chi^0 (\cdot, u)
\end{equation}
where $G$, $H$, $\alpha$ and $\chi^0$ are all smooth, under suitable
assumptions on $G$ and $\chi^0$. This will appear in~\cite{GL2}.

The paper is organized as follows. In Section~\ref{gblq-P} we
 recall some basic facts and formulas for Hermitian manifolds,
fixing notations along the way. We shall also construct some
special local coordinates; see Lemma~\ref{non-sym}.
These local coordinates are crucial to
our estimate of  $\Delta u$ in Section~\ref{gblq-G} where
we also derive gradient estimates,
extending the arguments of Blocki~\cite{Blocki09}
and P.-F. Guan~\cite{GuanPF} in the K\"ahler case.
Section~\ref{gblq-B} concerns the boundary estimates for second
derivatives. In Section~\ref{gblq-R} we come back to finish the
global estimates for all (real) second derivatives which enable us to
apply the Evans-Krylov theorem 
for $C^{2, \alpha}$ estimates; higher order estimates and regularity
then follow from the classical Schauder theory.
In Section~\ref{gblq-E} we discuss the $C^0$ estimates and
existence of solutions, completing the proof of
Theorem~\ref{gblq-th20} and Theorems~\ref{gblq-th30}-\ref{gblq-th50}.
Section~\ref{gblq-T} contains the proof of Theorem~\ref{gblq-TR}.
Finally, in Section~\ref{gblq-S} we discuss a Dirichlet problem
for the HCMA equation which is related to the Donaldson conjecture
in the K\"ahler case.

An earlier version (\cite{GL1}) of this article was posted on
the arXiv in June 2009. We learned afterwards of the work of
Cherrier and Hanani~\cite{Cherrier87}, \cite{Hanani96a}, \cite{Hanani96b}, 
\cite{CH98}, \cite{CH99}. We wish to thank Philippe Delano\"e for bringing 
these beautiful papers to our attention.
More recently, right before the current version was finished we received
from Tosatti and Weinkove their paper \cite{TW3} in which, among other
very interesting results, they were able to derive the $C^0$ estimates
for balanced Hermitian manifolds. We thank them for sending us the
preprint and for useful communications. Finally, the authors
wish to express their gratefulness to
Pengfei Guan and Fangyang Zheng for very helpful discussions and
suggestions.

\bigskip

\section{Preliminaries}
\label{gblq-P}
 \setcounter{equation}{0}
\medskip


 Let $M^n$ be a complex manifold of dimension $n$
and $g$ a Riemannian metric on $M$. Let $J$ be the induced
almost complex structure on $M$ so $J$ is integrable and
$J^2 = - \mbox{id}$. We assume that $J$ is compatible with $g$, i.e.
\begin{equation}
\label{cma-K10}
 g (u, v) = g (Ju, Jv), \;\; u, v \in TM;
\end{equation}
such $g$ is called a {\em Hermitian} metric.
Let $\omega$ be the {\em K\"ahler} form of $g$ defined
by
\begin{equation}
\label{cma-K20}
\omega (u, v) = - g (u, Jv).
\end{equation}
We recall that $g$ is {\em K\"ahler} if its K\"ahler form
$\omega$ is closed,  i.e. $d \omega = 0$.

The complex tangent bundle $T_C M = TM \times \bfC$ has a
natural splitting
\begin{equation}
\label{cma-K30}
 T_C M = T^{1,0} M + T^{0,1} M
\end{equation}
where $T^{1,0} M$ and  $T^{0,1} M$ are the $\pm
\sqrt{-1}$-eigenspaces of $J$. The metric $g$ is obviously extended
$\bfC$-linearly to $T_C M$, and
\begin{equation}
\label{cma-K40}
 g (u, v) = 0 \;\; \mbox{if $u,v \in T^{1,0} M$, or $u,v \in T^{0,1} M$}.
\end{equation}

Let $\nabla$ be the Chern connection of $g$. It satisfies
\begin{equation}
\label{cma-K55}
\nabla_u (g (v, w)) = g (\nabla_u v, w) + g (v, \nabla_u w)
\end{equation}
but may have nontrivial torsion.
The torsion $T$ and curvature $R$ of $\nabla$ are defined by
\begin{equation}
\label{cma-K95}
\begin{cases}
   T (u, v)  = \nabla_u v - \nabla_v u - [u,v], \\
 R (u, v) w  = \nabla_u \nabla_v w - \nabla_v \nabla_u w - \nabla_{[u,v]} w,
\end{cases}
\end{equation}
respectively.
Since $\nabla J = 0$ we have
\begin{equation}
\label{cma-K100}
g (R (u, v) Jw, Jx) = g (R (u, v) w, x) \equiv R (u, v, w, x).
\end{equation}
Therefore $R (u, v, w, x) = 0$ unless $w$ and $x$ are of different
type.

In local coordinates $(z_1, \ldots, z_n)$, we have
\begin{equation}
\label{cma-K60}
 J \frac{\partial}{\partial z_j} = \sqrt{-1}
\frac{\partial}{\partial z_j}, \;\;
 J \frac{\partial}{\partial \bz_j} = - \sqrt{-1}
\frac{\partial}{\partial \bz_j}.
\end{equation}
Thus, by \eqref{cma-K40}
\begin{equation}
\label{cma-K50}
 g \Big(\frac{\partial}{\partial z_j},
\frac{\partial}{\partial z_k}\Big) = 0, \;\;
  g \Big(\frac{\partial}{\partial \bz_j}, \frac{\partial}{\partial \bz_k}\Big)
   = 0.
\end{equation}
  We write
\begin{equation}
\label{cma-K70}
 g_{j\bk} = g \Big(\frac{\partial}{\partial z_j},
\frac{\partial}{\partial \bz_k}\Big),
  \;\; \{g^{i\bj}\} = \{g_{i\bj}\}^{-1}.
\end{equation}
That is, $g^{i\bj} g_{k\bj} = \delta_{ik}$.
 The K\"ahler form $\omega$ is given by
\begin{equation}
\label{cma-K80}
 \omega = \frac{\sqrt{-1}}{2} g_{j\bk} dz_j \wedge d \bz_k.
\end{equation}
For convenience we shall write
\[ \chi = \frac{\sqrt{-1}}{2} \chi_{j\bk} dz_j \wedge d \bz_k. \]

The Christoffel symbols $\Gamma^l_{jk}$ are defined by
\[ \nabla_{\frac{\partial}{\partial z_j}}
\frac{\partial}{\partial z_k}
   = \Gamma^l_{jk} \frac{\partial}{\partial z_l}. \]
Recall that by \eqref{cma-K55} and \eqref{cma-K50},
\begin{equation}
\label{cma-K85}
\left\{ \begin{aligned}
\nabla_{\frac{\partial}{\partial z_j}} \frac{\partial}{\partial \bz_k}
 \,& = \nabla_{\frac{\partial}{\partial \bz_j}} \frac{\partial}{\partial z_k}
     = 0, \\
 \nabla_{\frac{\partial}{\partial \bz_j}} \frac{\partial}{\partial \bz_k}
 \,& = \Gamma^{\bl}_{\bj \bk} \frac{\partial}{\partial \bz_l}
     = \ol{\Gamma^l_{jk}} \frac{\partial}{\partial \bz_l}
 \end{aligned} \right.
\end{equation}
and
\begin{equation}
\label{cma-K90}
\Gamma^l_{jk} = g^{l\bm} \frac{\partial g_{k\bm}}{\partial z_j}.
\end{equation}
The torsion is given by
\begin{equation}
\label{cma-K102}
 T^k_{ij} = \Gamma^k_{ij} - \Gamma^k_{ji} = g^{k\bl} \Big(\frac{\partial
g_{j\bl}}{\partial z_i} - \frac{\partial g_{i\bl}}{\partial z_j}\Big)
\end{equation}
while the curvature
\begin{equation}
\label{cma-K110}
\begin{aligned}
 R_{i\bj k\bl} \equiv &\, R \Big(\frac{\partial}{\partial z_i},
 \frac{\partial}{\partial \bz_j}, \frac{\partial}{\partial z_k},
 \frac{\partial}{\partial \bz_l}\Big) \\
 = \,& - g_{m \bl} \frac{\partial \Gamma_{ik}^m}{\partial \bz_j}
 = - \frac{\partial^2 g_{k\bl}}{\partial z_i \partial \bz_j}
       + g^{p\bq} \frac{\partial g_{k\bq}}{\partial z_i}
                  \frac{\partial g_{p\bl}}{\partial \bz_j}.
         \end{aligned}
\end{equation}
Note that from \eqref{cma-K100} and \eqref{cma-K90}
$R_{i\bj k l}  = R_{i j k l} = 0$
but, in general $R_{i j k \bl} \neq 0$.
By \eqref{cma-K110} and \eqref{cma-K102} we have
\begin{equation}
\label{cma-K115}
R_{i \bj k\bl} - R_{k \bj i\bl}
   =  g_{m \bl} \frac{\partial T_{ki}^m}{\partial \bz_j}
   = g_{m \bl} \nabla_{\bj} T_{ki}^m,
\end{equation}
which also follows form the general Bianchi identity.

The traces of the curvature tensor
\begin{equation}
\label{cma-K120'}
\begin{aligned}
 R_{k\bl} \,&= g^{i\bj} R_{i\bj k\bl}, \;\;
 S_{i\bj} = g^{k\bl} R_{i\bj k\bl}
   = -\frac{\partial^2}{\partial z_i \partial \bz_j} \log\det g_{k\bl}
\end{aligned} 
 \end{equation}
are called the {\em first} and {\em second} Ricci tensors, respectively.
Therefore one can consider extensions of Calabi-Yau theorem for $S_{i\bj}$;
see \cite{TW3}.

The following special local coordinates, which will be used in our proof of
{\em a priori} estimates for $|\nabla u|$ and $\Delta u$, seems of interest
in itself.

\begin{lemma}
\label{non-sym}
Around a point $p \in M$ there exist local coordinates such that, at $p$,
\begin{equation}
\label{coord}
\begin{aligned}
   & g_{i\bj} = \delta_{ij},   \;\;
    \frac{\partial g_{i\bi}}{\partial z_j} = 0, \;\; \forall \;i, j.
\end{aligned}
\end{equation}
\end{lemma}

\begin{proof}
Let $(z_1, z_2, \cdots, z_n)$ be a local coordinate system around $p$
such that $z_i(p)=0$ for $i=1, \cdots, n$ and
\[ g_{i\bj}(p) := g \Big(\f {\p}{\p z_i}, \f {\p}{\p \bar{z_j}}\Big)
    = \d_{i j}. \]
Define new coordinates $(w_1, w_2, \cdots, w_n)$ by
\begin{equation}
\label{new-coord}
w_r = z_r + \sum_{m \neq r} \f {\p g_{r\br}}{\p z_m}(p)z_m z_r
          + \frac{1}{2} \f {\p g_{r\br}}{\p z_r}(p) z_r^2,
\;\; 1 \leq r \leq n.
\end{equation}
We have
\begin{equation}\label{new-metric}
\tg_{i\bj}: =g \Big(\f {\p}{\p w_i}, \f {\p}{\p \bar{w_j}}\Big)
 =\sum_{r,s} g_{r\bs} \f {\p z_r}{\p w_i} \overline{\f {\p z_s}{\p w_j}}.
\end{equation}
It follows that
\begin{equation}
\label{ijj}
\f {\p \tg_{i\bj}}{\p w_k} =
  \sum_{r,s} g_{r\bs} \f {\p^2 z_r}{\p w_i \p w_k}
    \overline{\f {\p z_s}{\p w_j}}
  + \sum_{r,s, p} \f {\p g_{r \bs}}{\p z_p}
    \f {\p z_p}{\p w_k} \f {\p z_r}{\p w_i}  \overline{\f {\p z_s}{\p w_j}}.
\end{equation}

Differentiate (\ref{new-coord}) with respect to $w_i$ and $w_k$.
We see that, at $p$,
\[ \f {\p z_r}{\p w_i}=\d_{r i}, \;\;
\f {\p^2 z_r}{\p w_i \p w_k} = - \sum_{m \neq r} \f {\p g_{r\br}}{\p z_m}
   \Big(\f {\p z_m}{\p w_i} \f {\p z_r}{\p w_k}
      + \f {\p z_m}{\p w_k}\f {\p z_r}{\p w_i}\Big)
      - \f {\p g_{r\br}}{\p z_r} \f {\p z_r}{\p w_i} \f {\p z_r}{\p w_k}. \]
Plugging these into (\ref{ijj}), we obtain at $p$,
\begin{equation}
\label{gblq-P80}
\left\{
\begin{aligned}
\frac{\partial \tg_{i\bi}}{\partial w_k}
   = \, & \frac{\partial g_{i\bi}}{\partial z_k}
      - \frac{\partial g_{i\bi}}{\partial z_k} = 0, \;\; \forall \; i, k,\\
\frac{\partial \tg_{i\bj}}{\partial w_j}
   = \,& \frac{\partial g_{i\bj}} {\partial z_j}
      - \frac{\partial g_{j\bj}}{\partial z_i} = T_{ji}^j,
      \;\; \forall \; i \neq j, \\
\frac{\partial \tg_{i\bj}}{\partial w_k}
   = \,& \frac{\partial g_{i\bj} }{\partial z_k}, \;\; \mbox{otherwise.}
\end{aligned}
\right.
\end{equation}
Finally, switching $w$ and $z$ gives \eqref{coord}.
\end{proof}

\begin{remark}
If, in place of \eqref{new-coord}, we
define
\begin{equation}
\label{new-coord'}
w_r = z_r + \sum_{m \neq r} \f {\p g_{m\br}}{\p z_r}(p)z_m z_r
          + \frac{1}{2} \f {\p g_{r\br}}{\p z_r}(p) z_r^2,
\;\; 1 \leq r \leq n,
\end{equation}
then under the new coordinates $(w_1, w_2, \cdots, w_n)$,
\begin{equation}
\label{gblq-P80'}
\left\{
\begin{aligned}
\frac{\partial \tg_{i\bj}}{\partial w_j} (p)
   = \, & 0, \;\; \forall \; i, j,\\
\frac{\partial \tg_{i\bi}}{\partial w_k} (p)
      = \,& T_{ki}^i, \;\; \forall \; i \neq k, \\
\frac{\partial \tg_{i\bj}}{\partial w_k} (p)
   = \,& \frac{\partial g_{i\bj} }{\partial z_k} (p), \;\; \mbox{otherwise.}
 \end{aligned}
\right.
\end{equation}
\end{remark}

In \cite{ST} Streets and Tian constructed local coordinates
\begin{equation}
\label{gblq-G100}
    g_{i\bj} = \delta_{ij},  \;\;
   \frac{\partial g_{i\bj}}{\partial z_k}
      + \frac{\partial g_{k\bj}}{\partial z_i} = 0
\end{equation}
and consequently, $T_{ij}^k = 2 \frac{\partial g_{j\bk}}{\partial z_i}$ at
a fixed point.
In general it is impossible to find local coordinates satisfying
both \eqref{coord} and \eqref{gblq-G100} simultaneously.

Let $\Lambda^{p, q}$ denote differential forms of type $(p, q)$ on $M$.
The exterior differential $d$ has a natural decomposition
$d = \partial + \bpartial$ where
\[ \partial: \Lambda^{p, q} \rightarrow \Lambda^{p+1, q} , \;\;
  \bpartial: \Lambda^{p, q} \rightarrow \Lambda^{p, q+1}.  \]
Recall that $\partial^2 = \bpartial^2 = \partial \bpartial + \bpartial \partial = 0$
and, by the Stokes theorem
\[ \int_M \partial \alpha = \int_{\partial M} \alpha,
 \;\; \forall \; \alpha \in \Lambda^{n-1, n}. \]
A similar formula holds for $\bpartial$.

For a function $u \in C^2 (M)$, $\partial \bpartial u$
is given in local coordinates by
\begin{equation}
\label{cma-K80'}
\partial \bpartial u = \frac{\partial^2 u}{\partial z_i \partial \bz_j}
  dz_i \wedge d \bz_j.
\end{equation}
 Equation~\eqref{gblq-I10} thus takes the form
\begin{equation}
\label{cma2-M10}
 \det \Big(\chi_{i\bj} + \frac{\partial^2 u}{\partial z_i \partial \bz_j}\Big)
    = \psi (z, u) \det g_{i\bj}.
\end{equation}

We use $\nabla^2 u$ to denote the {\em Hessian} of $u$:
\begin{equation}
\label{cma-K215}
\nabla^2 u (X, Y) \equiv \nabla_Y \nabla_X u
   = Y (X u) - (\nabla_Y X) u, \;\; X, Y \in TM.
\end{equation}
By \eqref{cma-K85} we see that
\begin{equation}
\label{cma-K220}
\nabla_{\frac{\partial}{\partial z_i}}
 \nabla_{\frac{\partial}{\partial \bz_j}} u
   = \frac{\partial^2 u}{\partial z_i \partial \bz_j}.
\end{equation}
Consequently, the Laplacian of $u$ with respect to the Chern connection is
\begin{equation}
\label{cma-K230}
\Delta u = g^{i\bj} \frac{\partial^2 u}{\partial z_i \partial \bz_j},
\end{equation}
or equivalently,
\begin{equation}
\label{cma-K235}
\Delta u \omega^n = \frac{\sqrt{-1}}{2} \partial \bpartial u \wedge \omega^{n-1}.
\end{equation}
Integrating \eqref{cma-K235} (by parts), we obtain
\begin{equation}
\label{cma-K238}
\begin{aligned}
\frac{2}{\sqrt{-1}} \int_M \Delta u \omega^n
  = \,& \int_M \partial \bpartial u \wedge \omega^{n-1} \\
  = \,& \int_{\partial M} \bpartial u \wedge \omega^{n-1}
        + \int_M  \bpartial u \wedge \partial \omega^{n-1} \\
  = \,& \int_{\partial M} (\bpartial u \wedge \omega^{n-1} + u \partial \omega^{n-1})
   + \int_M  u \partial \bpartial \omega^{n-1}.
   \end{aligned}
\end{equation}

\bigskip

\section{Global estimates for $|\nabla u|$ and $\Delta u$}
\label{gblq-G}
\setcounter{equation}{0}
\medskip

Let $u \in \cH_{\chi} \cap C^4 (\bM)$ be a solution of (\ref{cma2-M10}).
In this section we derive the following estimates
\begin{equation}
\label{gblq-I50}
\max_{\bM} |\nabla u| \leq C_1 \left(1 + \max_{\partial M} |\nabla u|\right),
\end{equation}
\begin{equation}
\label{gblq-I60}
\max_{\bM} \Delta u\leq C_2 \left(1 + \max_{\partial M} \Delta u\right).
\end{equation}
Here we emphasize that $C_1$ and $C_2$ depend only on geometric quantities
(torsion and curvature) of $M$ and on $\chi$ as well as its covariant
derivatives, but do not depend on
$\inf \psi$ so the estimates \eqref{gblq-I50} and \eqref{gblq-I60} apply
to the degenerate case ($\psi \geq 0$);
see Propositions~\ref{gblq-prop-G10} and \ref{gblq-prop-C10}  for details.

For $\chi = \omega$ these estimates were derived by
Cherrier and Hanani~\cite{Hanani96a}, \cite{Hanani96b}, \cite{CH98}, \cite{CH99}.
The estimate for $\Delta u$ is an extension of that of Yau~\cite{Yau78}.
The gradient estimate \eqref{gblq-I50} was also independently recovered by
Blocki~\cite{Blocki09} and P.-F. Guan~\cite{GuanPF} in the K\"ahler case,
and by Xiangwen Zhang~\cite{Zhang} for more general equations on compact
Hermitian manifolds without boundary.

We shall first assume $\chi = \chi (\cdot, u)$ is positive definite. More precisely,
\begin{equation}
\label{gblq-G1}
 \chi = \chi (\cdot, u) \geq \epsilon \omega
\end{equation}
 where $\epsilon > 0$ may depend on $\sup_M |u|$. In this case we do not need the
 subsolution $\ul{u}$ to derive \eqref{gblq-I50} and \eqref{gblq-I60}.
 At the end of this section we shall remove assumption~\eqref{gblq-G1}.

 Throughout this section we use ordinary derivatives.
For convenience we write in local coordinates,
\[ u_i = \frac{\partial u}{\partial z_i}, \;\;
   u_{\bi} = \frac{\partial u}{\partial \bz_i}, \;\;
   u_{i\bj} = \frac{\partial^2 u}{\partial z_i \partial \bz_j}, \;\;
   g_{i\bj k} = \frac{\partial g_{i\bj}}{\partial z_k},  \;\;
   g_{i\bj k\bl} = \frac{\partial^2 g_{i\bj}}{\partial z_k \partial \bz_l},
 \;\; \mbox{etc}, \]
and $\fg_{i\bj} = u_{i\bj} + \chi_{ij}$,
$\{\fg^{i\bj}\} = \{\fg_{i\bj}\}^{-1}$.

Suppose at a fixed point $p \in M$,
\begin{equation}
\label{gblq-G110}
 g_{i\bj} = \delta_{ij} \; \mbox{and $\{\fg_{i\bj}\}$ is diagonal.}
\end{equation}
Starting from
\begin{equation}
\label{gblq-G10}
|\nabla u|^2 = g^{k\bl} u_k u_{\bl}, \;\;
 \Delta u  = g^{k\bl} u_{k\bl},
\end{equation}
by straightforward calculation, we see that
\begin{equation}
\label{gblq-C30}
\begin{aligned}
 (|\nabla u|^2)_i =\,&  u_k u_{k\bi} + (u_{ki}  - g_{k\bl i} u_{l}) u_{\bk},\\
     (\Delta u)_i =\,&  u_{k\bk i} -  g_{l\bk i} u_{k \bl},
\end{aligned}
\end{equation}
\begin{equation}
\label{cma2-M60}
\begin{aligned}
(|\nabla u|^2)_{i\bi} = \,
  & u_{i\bk} u_{k\bi} + u_{i\bi k} u_{\bk} + u_{i \bi \bk} u_k
     + |u_{ki} - g_{k\bl i} u_{l}|^2  \\
  &  - 2 \fRe \{g_{k\bl i} u_{l \bi} u_{\bk}\}
     + (g_{l\bp i} g_{p\bk \bi} - g_{l \bk i \bi}) u_k u_{\bl}.
\end{aligned}
\end{equation}
\begin{equation}
\label{gblq-C40}
\begin{aligned}
(\Delta u)_{i\bi} = \,
  & u_{k\bk i\bi} - 2 \fRe \{u_{k \bj i} g_{j \bk \bi}\}
    + (g_{l\bp i} g_{p\bk \bi} + g_{p\bk i} g_{l\bp \bi}
        -  g_{l \bk i \bi}) u_{k\bl} \\
= \, & u_{i\bi k\bk} - 2 \fRe \{u_{k \bj i} g_{j \bk \bi}\}
 + (\fg_{k\bl} - \chi_{k\bl} ) (g_{p\bk i} g_{l\bp \bi} + R_{i\bi k\bl}).
\end{aligned}
\end{equation}

Next, differentiate equation (\ref{cma2-M10}) twice,
\begin{equation}
\label{gblq-M70}
\begin{aligned}
\fg^{i\bi} u_{i\bi k} u_{\bk} 
    = & \, |\nabla u|^2 f_u + f_{z_k} u_{\bk}
      + g_{i\bi k} u_{\bk} - \fg^{i\bi} (\chi_{i\bi})_k u_{\bk},
       \end{aligned}
\end{equation}
\begin{equation}
\label{gblq-C70}
\begin{aligned}
\fg^{i\bi} u_{i\bi k\bk}  = \,
        & \fg^{i\bi} \fg^{j\bj} |u_{i\bj k} + (\chi_{i\bj})_k|^2
          + (f)_{k\bk} - R_{k\bk i\bi}
          - \fg^{i\bi} (\chi_{i\bi})_{k \bk}.
       \end{aligned}
\end{equation}
Note that
\begin{equation}
\label{gblq-C270}
(\chi_{i\bj})_k = \nabla_k \chi_{i\bj} + g_{i\bl k} \chi_{l\bj}
                     +  u_k D_u \chi_{i\bj}.
 \end{equation}
Thus,
\[ g_{k\bl i} u_{l \bi} + (\chi_{i\bi})_k
= g_{k\bi i} \fg_{i\bi} + \nabla_k \chi_{i\bi} + T_{ik}^l \chi_{l\bi}
                     + u_k D_u \chi_{i\bi}. \]
From (\ref{cma2-M60}) and (\ref{gblq-M70}) we see that
\begin{equation}
\label{bglq-M80}
\begin{aligned}
\fg^{i\bi} (|\nabla u|^2)_{i\bi} =
\, & \fg^{i\bi} u_{i\bk} u_{k\bi}
     + \fg^{i\bi} |u_{ki} - g_{k\bl i} u_{l}|^2
     + \fg^{i\bi} R_{i \bi k \bl} u_l u_{\bk} \\
   & -  2 |\nabla u|^2 \fg^{i\bi} D_u \chi_{i\bi}
     - 2 \fg^{i\bi} \fRe\{(\nabla_k \chi_{i\bi}
                           + T_{ki}^l \chi_{l\bi}) u_{\bk}\} \\
   & + 2 |\nabla u|^2 f_u + 2 \fRe \{(f_{z_k} - g_{k\bi i})u_{\bk}\}.
\end{aligned}
\end{equation}


\begin{proposition}
\label{gblq-prop-G10}
There exists $C_1 > 0$ depending on
$\sup_{M} |u|$, $|\psi^{\frac{1}{n}}|_{C^1}$
and
 \begin{equation}
\label{cma-35}
 \sup_{M} |\chi|, \; \sup_{M} |\nabla \chi|, \;
\inf_M \inf_i D_u \chi_{i\bi}, \; \sup_{M} |T|, \;
\inf_M \inf_{j, l} R_{j\bj l\bl}
\end{equation}
such that \eqref{gblq-I50} holds.
\end{proposition}

\begin{proof}
Let $L = \inf_M u$ and $\phi =A e^{L-u}$ where $A > 0$ is constant
to be determined. Suppose that $e^{\phi} |\nabla u|^2$ attains its
maximum at an interior point $p \in M$ where all calculations are
done in this proof.
By Lemma~\ref{non-sym} we assume \eqref{gblq-G110} and \eqref{coord} hold
at $p$ so that $g_{k\bi i} = T_{ik}^i$.
Since $e^{\phi} |\nabla u|^2$ attains its maximum at $p$,
\begin{equation}
\label{gblq-G30}
  \frac{(|\nabla u|^2)_i}{|\nabla u|^2} + \phi_i = 0, \;\;
\frac{(|\nabla u|^2)_{\bi}}{|\nabla u|^2} + \phi_{\bi} = 0
\end{equation}
and
\begin{equation}
\label{gblq-G40}
 \frac{(|\nabla u|^2)_{i\bi}}{|\nabla u|^2}
-  \frac{|(|\nabla u|^2)_{i}|^2}{|\nabla u|^4}
  + \phi_{i\bi} \leq 0.
\end{equation}
By \eqref{gblq-C30} and \eqref{gblq-G30},
\begin{equation}
\label{gblq-G50}
 |(|\nabla u|^2)_i|^2 = \sum_k |u_{k}|^2 |u_{ki}  - g_{k\bl i} u_{l}|^2
 - 2 |\nabla u|^2 \fRe \{u_i u_{i\bi} \phi_{\bi}\} - |u_i|^2  u_{i\bi}^2.
\end{equation}
Combining \eqref{gblq-G40}, \eqref{gblq-G50} and \eqref{bglq-M80}, we obtain
\begin{equation}
\label{bglq-M20}
\begin{aligned}
 |\nabla u|^2 \fg^{i\bi} \phi_{i\bi}
 \, &  - 2  \fRe \{\fg^{i\bi} u_i u_{i\bi} \phi_{\bi}\} \\
    &\, \leq 2 |\nabla f| |\nabla u| + |\nabla u|^2 (C - 2f_u)
      + C |\nabla u|^2 \sum \fg^{i\bi}
  \end{aligned}
\end{equation}
where $C$ depends on the quantities in \eqref{cma-35}.

 We have $\phi_i = - \phi u_i$, $\phi_{i\bi} = \phi (u_i u_{\bi} - u_{i\bi})$.
Therefore,
\begin{equation}
\label{gblq-G60}
\begin{aligned}
\fg^{i\bi}  \fRe\{u_i u_{i\bi} \phi_{\bi}\}
  =  - \phi |\nabla u|^2 + \phi \fg^{i\bi} \chi_{i\bi} u_{i} u_{\bi}
\geq - \phi |\nabla u|^2
     \end{aligned}
\end{equation}
and, by assumption~\eqref{gblq-G1},
\begin{equation}
\label{gblq-G70}
\begin{aligned}
\fg^{i\bi} \phi_{i\bi}
  = \, &  \phi \fg^{i\bi} u_{i} u_{\bi}
          - \phi + \phi \fg^{i\bi} \chi_{i\bi} \\
 \geq \, & - \phi + \phi \fg^{i\bi} u_{i} u_{\bi}
+ \epsilon \phi \sum \fg^{i\bi}.
      \end{aligned}
\end{equation}
 Note that
 \begin{equation}
\label{gblq-G80}
\begin{aligned}
 \fg^{i\bi} u_{i} u_{\bi} + \epsilon \sum \fg^{i\bi}
 \geq \, & |\nabla u|^2 \min_{i} \fg^{i\bi}
         + \epsilon  \sum \fg^{i\bi} \\
 \geq & \,  n \epsilon^{\frac{n-1}{n}}
          |\nabla u|^{\frac{2}{n}} (\det \fg^{i\bj})^{\frac{1}{n}}.
     \end{aligned}
\end{equation}
Thus
\begin{equation}
\label{bglq-M20'}
\begin{aligned}
n \epsilon^{\frac{n-1}{n}}
\psi^{-\frac{1}{n}} |\nabla u|^{\frac{2}{n}}
          \,& + \fg^{i\bi} u_{i} u_{\bi} + \epsilon \sum \fg^{i\bi}
 \leq 2 (1 + \phi^{-1} \fg^{i\bi} \phi_{i\bi}) \\
 \leq \,& 2 + 4 (|\nabla f| |\nabla u|^{-1} - f_u) \phi^{-1}
          + C_1 \phi^{-1} \sum \fg^{i\bi} \\
        & + (2^{-1} + C_2 \phi^{-1}) \fg^{i\bi} u_{i} u_{\bi}.
\end{aligned}
\end{equation}
Choose $A$ sufficiently large so that $\epsilon \phi \geq C_1$ and
$\phi \geq 2 C_2$. We see that
\[ |\nabla u|^{\frac{2}{n}}
    \leq C_3 (\psi^{\frac{1}{n}} +
          D_u \psi^{\frac{1}{n}} + |\nabla \psi^{\frac{1}{n}}|). \]
 This proves \eqref{gblq-I50}.
\end{proof}

\begin{lemma}
\label{gblq-lemma-C10}
Assume that \eqref{coord} and \eqref{gblq-G110} hold at $p \in M$ .
Then at $p$,
\begin{equation}
\label{gblq-C120}
\begin{aligned}
\fg^{i\bi} (\Delta u)_{i\bi} \geq
\, &  \fg^{i\bi} \fg^{j\bj} |u_{i\bj j} + (\chi_{i\bj})_j|^2
       + \Delta (f) - n^2 \inf_{j,k} R_{j\bj k\bk} \\
   &  + c_1 (\Delta u + \mbox{tr} \chi)\sum \fg^{i\bi}
   - (c_2+ c_3 |\nabla u|^2) \sum \fg^{i\bi}
       \end{aligned}
\end{equation}
where $c_1 = \inf R_{j\bj k\bk} - \sup D_u \chi_{i\bi}$ and
$c_3 = 0$ if $\chi$ does not depend on $u$.
\end{lemma}

\begin{proof}

Write
\begin{equation}
\label{gblq-C100}
 u_{k \bj i} g_{j \bk \bi}
 = [u_{i \bj k} + (\chi_{i\bj})_k] g_{j \bk \bi}
 - (\chi_{i\bj})_k g_{j \bk \bi}.
\end{equation}
By Cauchy-Schwarz inequality,
\begin{equation}
\label{gblq-C110}
\begin{aligned}
 2 \sum_{j \neq k} |\fRe \{[u_{i \bj k} + (\chi_{i\bj})_k] g_{j \bk \bi}\}|
 \leq \, & \sum_{j \neq k} \fg^{j\bj} |u_{i \bj k} + (\chi_{i\bj})_k|^2
 + \sum_{j \neq k} \fg_{j\bj} |g_{k\bj i}|^2.
   \end{aligned}
\end{equation}
Combining  \eqref{gblq-C40},  \eqref{gblq-C70}, \eqref{gblq-C100},
and \eqref{gblq-C110},
we derive
\begin{equation}
\label{gblq-C120'}
\begin{aligned}
\fg^{i\bi} (\Delta u)_{i\bi} \geq
\, &  \fg^{i\bi} \fg^{j\bj} |u_{i\bj j} + (\chi_{i\bj})_j|^2
     + \Delta (f) - \fg^{i\bi} G_{i\bi} \\
   & + (\fg^{i\bi} \fg_{j\bj} - 1) R_{i\bi j\bj}
     - \fg^{i\bi} R_{i\bi l\bk} \chi_{l\bk}
       \end{aligned}
\end{equation}
where
\[ G_{i\bi}= g_{j\bl i} g_{k\bj \bi} \chi_{l\bk}
       - 2 \fRe\{(\chi_{i\bj})_k g_{j \bk \bi}\} + (\chi_{i\bi})_{k\bk}. \]
Next,
\[ (\chi_{i\bi})_{k\bk} =  \chi_{i\bi k\bk}
      + 2 \fRe\{u_{\bk} (D_u \nabla_k \chi_{i\bi} + g_{i\bl k} D_u \chi_{l\bi})\}
      + u_k u_{\bk} D^2_u \chi_{i\bi} + u_{k\bk} D_u \chi_{i\bi}\]
where
\[ \chi_{i\bi k\bk}
  = \nabla_{\bk} \nabla_k \chi_{i\bi} - R_{k\bk i\bl} \chi_{l\bi}
    + g_{i\bl k} g_{m\bi \bk} \chi_{l\bm}
    + 2 \fRe\{g_{i\bj k} \nabla_{\bk} \chi_{j\bi}\}. \]
It follows from \eqref{gblq-C270} and the identity
\[ |g_{i\bj k}|^2 + |g_{k\bj i}|^2 - 2 \fRe \{g_{i\bj k} g_{j \bk \bi}\}
  = |g_{k\bj i} - g_{i\bj k}|^2 = |T_{ik}^j|^2 \]
that
\begin{equation}
\begin{aligned}
 G_{i\bi} =   & \nabla_{\bk} \nabla_k \chi_{i\bi}
     - R_{k\bk i\bj} \chi_{j\bi} + T_{ik}^j  T_{ik}^l \chi_{j\bl}
     - 2 \fRe \{T_{ik}^j \nabla_{\bk} \chi_{j\bi}\} \\
   & + 2 \fRe \{u_{\bk} \nabla_k D_u \chi_{i\bi}
     - T_{ik}^j u_{\bk} D_u \chi_{j\bi}\} \\
   & + |\nabla u|^2 D^2_u \chi_{i\bi} + \Delta u  D_u \chi_{i\bi}.
       \end{aligned}
\end{equation}
Plugging this into \eqref{gblq-C120'} we prove \eqref{gblq-C120}.
\end{proof}

\begin{proposition}
\label{gblq-prop-C10}
There exists constant $C_2 > 0$ depending on
\[ \mbox{$ |u|_{C^1 (\bM)}$, 
$|\psi^{\frac{1}{n-1}}|_{C^2 (\bM \times \bfR)}$,
$|\chi|_{C^2 (\bM \times \bfR)}$} \]
and the geometric quantities (curvature and torsion) of $M$,
such that \eqref{gblq-I60} holds.
If both $\chi$ and $\psi$ are independent of $u$, then $C_2$ does not depend
on $|\nabla u|_{C^0 (\bM)}$.
\end{proposition}

\begin{proof}
Let $a = \sup \mbox{tr} \chi$ and consider the function
$\varPhi = e^\phi (a + \Delta u)$ where $\phi = A e^{L -u}$
and $L = inf_M u$ as in the proof of Proposition~\ref{gblq-prop-G10}.
Suppose $\varPhi$ achieves its maximum at an interior point $p \in M$
where we assume \eqref{coord} and \eqref{gblq-G110} hold.
 We have (all calculations are done at $p$ below)
\begin{equation}
\label{gblq-C80}
  \frac{(\Delta u )_i}{a + \Delta u} + \phi_i = 0, \;\;
\frac{(\Delta u)_{\bi}}{a + \Delta u} + \phi_{\bi} = 0,
\end{equation}
\begin{equation}
\label{gblq-C90}
 \frac{(\Delta u)_{i\bi}}{a + \Delta u}
-  \frac{|(\Delta u)_{i}|^2}{(a + \Delta u)^2}
  + \phi_{i\bi} \leq 0.
\end{equation}
Write
\[ (\Delta u)_{i} = u_{i\bj j} + (\chi_{i\bj})_j + \lambda_i \]
where, by \eqref{gblq-C30}, \eqref{gblq-G110}, \eqref{coord} and
\eqref{gblq-C270},
\[ \lambda_i = - (\chi_{i\bj})_j -  g_{l\bk i} u_{k\bl}
             =  - \nabla_j \chi_{i\bj} +  T_{ij}^l \chi_{l\bj}
               - u_j D_u \chi_{i\bj}. \]
We have by \eqref{gblq-C80},
\begin{equation}
\label{gblq-C130}
\begin{aligned}
|(a + \Delta u)_{i}|^2 \,
= \, & |u_{i\bj j} + (\chi_{i\bj})_j|^2
+ 2 \fRe \{(\Delta u)_{i} \bar{\lambda_{i}}\} - |\lambda_i|^2 \\
= |u_{i\bj j} \, &  + (\chi_{i\bj})_j|^2
  - 2 (a + \Delta u) \fRe\{\phi_{i} \bar{\lambda_{i}}\} - |\lambda_i|^2.
 \end{aligned}
\end{equation}
By Cauchy-Schwarz inequality,
\begin{equation}
\label{gblq-C140}
  \begin{aligned}
 \fg^{i\bi} |u_{i\bj j} + (\chi_{i\bj})_j|^2
   = & \, \fg^{i\bi} 
     |\fg_{j\bj}^{1/2} \; \fg_{j\bj}^{- 1/2} \;
      (u_{i\bj j} + (\chi_{i\bj})_j)| \\
\leq   & \; (\mbox{tr} \chi + \Delta u)  \fg^{i\bi} \fg^{j\bj}
      |u_{i\bj j} + (\chi_{i\bj})_j|^2.
 \end{aligned}
\end{equation}
From \eqref{gblq-C90}, \eqref{gblq-C120}, \eqref{gblq-C130} and
\eqref{gblq-C140} we derive
\begin{equation}
\label{gblq-C150}
  \begin{aligned}
 (a + \Delta u) \fg^{i\bi} \phi_{i\bi}
      \,& +  2 \fg^{i\bi} \fRe\{\phi_{i} \bar{\lambda_{i}}\} \\
 \leq \,& - c_1 (\mbox{tr} \chi + \Delta u) \sum \fg^{i\bi}
           - \Delta (f) \\
        &  + n^2 \inf_{j,k} R_{j\bj k\bk}
         + (c_2+ c_3 |\nabla u|^2) \sum \fg^{i\bi}.
 \end{aligned}
\end{equation}

By \eqref{gblq-G70} and using the following inequality as in Yau~\cite{Yau78}
\begin{equation}
\label{gblq-C190}
 \Big(\sum \fg^{i\bi}\Big)^{n-1}
\geq \frac{\sum \fg_{i\bi}}{\fg_{1\bar{1}} \cdots \fg_{n\bar{n}}}
   = \frac{\mbox{tr} \chi + \Delta u}{\det (\chi_{i\bj} + u_{i\bj})}
   = \frac{\mbox{tr} \chi + \Delta u}{\psi},
  \end{equation}
we see that
\begin{equation}
\label{gblq-C150'}
  \begin{aligned}
\phi^{-1} \fg^{i\bi} \phi_{i\bi} \geq
   \frac{\epsilon}{2 \psi^{\frac{1}{n-1}}}
    (\mbox{tr} \chi + \Delta u)^{\frac{1}{n-1}} - 1
  + \frac{\epsilon}{2} \sum \fg^{i\bi}
  + \fg^{i\bi} u_i u_{\bi}.
 \end{aligned}
\end{equation}
Note also that
\[ 2 \fg^{i\bi} \fRe\{\phi_{i} \bar{\lambda_{i}}\}
= 2 \phi \fg^{i\bi} \fRe\{u_{i} \bar{\lambda_{i}}\}
\geq - \phi \fg^{i\bi} (u_i u_{\bi} + |\lambda_i|^2). \]
Consequently, when $A$ is chosen so that $\epsilon \phi \geq 2 (1 - c_1)$
we have from \eqref{gblq-C150}, \eqref{gblq-C150'} and \eqref{gblq-C190},
\[ (\mbox{tr} \chi + \Delta u)^{\frac{1}{n-1}}
      \leq C (1 + |\psi^{\frac{1}{n-1}}|_{C^2}) \]
or
\[ a + \Delta u \leq A |\lambda|^2 + c_2+ c_3 |\nabla u|^2. \]
The proof is complete.
\end{proof}

Finally, to remove assumption~\eqref{gblq-G1} we need assume $\chi (\cdot, u)$ to be
nondecreasing in $u$. In this case $\chi (\cdot, u) - \chi (\cdot, \ul{u}) \geq 0$.
We may therefore replace $u$ by $v = u - \ul{u}$ and $\chi (\cdot, u)$ by
\[ \chi' (\cdot, v) \equiv
    \chi (\cdot, v + \ul{u}) + \frac{\sqrt{-1}}{2} \partial \bpartial \ul{u}. \]
Note that $\chi' (\cdot, v) \geq \chi_{\ul{u}} (\cdot, \ul{u}) > 0$
and $\chi'_v (\cdot, v) = \chi_u (\cdot, u)$.
Propositions~\ref{gblq-prop-G10} and \ref{gblq-prop-C10} thus apply to $v$.

\bigskip

\section{Boundary estimates for second derivatives}
\label{gblq-B}
\setcounter{equation}{0}
\medskip

In this section we derive
{\em a priori} estimates for second derivatives (the {\em real} Hessian) on the boundary
\begin{equation}
 \label{cma-37}
\max_{\partial M} |\nabla^2 u| \leq C.
\end{equation}
We shall follow techniques developed in \cite{GS93},
\cite{Guan98a}, \cite{Guan98b} using subsolutions.

We begin with a brief review of formulas for covariant
derivatives which we shall use in this and following sections.

In local coordinates $z = (z_1, \ldots, z_n)$,
$z_j = x_j + \sqrt{-1} y_j$, we use notations such as
\[ v_i = \nabla_{\frac{\partial}{\partial z_i}} v, \;
   v_{ij} = \nabla_{\frac{\partial}{\partial z_j}}
             \nabla_{\frac{\partial}{\partial z_i}} v, \;
   v_{x_i} = \nabla_{\frac{\partial}{\partial x_i}} v, \; \mbox{etc.} \]
Recall that
\begin{equation}
\label{gblq-B105}
 \begin{aligned}
& v_{i\bj} - v_{\bj i}  = 0, \;\;
  v_{ij} - v_{ji}  = T_{ij}^l v_l.
\end{aligned}
\end{equation}
By straightforward calculations,
\begin{equation}
\label{gblq-B145}
\left\{
\begin{aligned}
v_{i \bj \bk} - v_{i \bk \bj} = \, & \ol{T_{jk}^l} v_{i\bl},  \\
 v_{i \bj k} - v_{i k \bj} = \, & - g^{l\bm} R_{k \bj i \bm} v_l, \\
 v_{i j k} - v_{i k j} = \, & g^{l\bm} R_{jki\bm} v_l + T_{jk}^l v_{il}.
\end{aligned}
 \right.
\end{equation}
Therefore,
\begin{equation}
\label{gblq-B150}
 \begin{aligned}
 v_{i \bj k} - v_{k i \bj}
  = \,& (v_{i \bj k} - v_{i k \bj}) + (v_{i k \bj} - v_{k i \bj}) \\
  = \,& - g^{l\bm} R_{k \bj i \bm} v_l + T_{ik}^l v_{l\bj}
         + \nabla_{\bj} T_{ik}^l v_l \\
  = \,& - g^{l\bm} R_{i \bj k \bm} v_l + T_{ik}^l v_{l\bj}
\end{aligned}
\end{equation}
by \eqref{cma-K115},
and
\begin{equation}
\label{gblq-B155}
 \begin{aligned}
 v_{ijk} - v_{kij}
   =  \,& (v_{ijk} - v_{ikj}) + (v_{ikj} - v_{kij}) \\
   =  \,&  g^{l\bm} R_{jk i \bm} v_l + T_{jk}^l v_{il}
          + T_{ik}^l  v_{l j}  + \nabla_{j} T_{ik}^l v_l.
\end{aligned}
\end{equation}

Since
\[ \frac{\partial}{\partial x_k}
     = \frac{\partial}{\partial z_k} + \frac{\partial}{\partial \bz_k}, \;\;
   \frac{\partial}{\partial y_k}
 = \sqrt{-1} \Big(\frac{\partial}{\partial z_k}
  - \frac{\partial}{\partial \bz_k}\Big), \]
we see that
\begin{equation}
\label{gblq-B110}
\left\{
\begin{aligned}
v_{z_i x_j} -  v_{x_j z_i} = \,&  v_{ij} - v_{ji}  = T_{ij}^l v_l,  \\
v_{z_i y_j} -  v_{y_j z_i} = \,& \sqrt{-1} (v_{ij} - v_{ji})
                           = \sqrt{-1} T_{ij}^l v_l,
\end{aligned} \right.
\end{equation}
\begin{equation}
\label{gblq-B120}
\begin{aligned}
v_{z_i \bz_j x_k} - v_{x_k z_i \bz_j}
     = \,&  (v_{i\bj k} + v_{i\bj \bk}) - (v_{ki\bj} + v_{\bk i \bj}) \\
     = \,&  (v_{i\bj k} - v_{ki\bj}) + (v_{i\bj \bk} - v_{i \bk \bj}) \\
    = \,& - g^{l\bm} R_{i \bj k \bm} v_l + T_{ik}^l v_{l\bj}
          + \ol{T_{jk}^l} v_{i\bl}.
  \end{aligned} 
\end{equation}
Similarly,
\begin{equation}
\label{gblq-B130}
\begin{aligned}
v_{z_i z_{\bj} y_k} -  v_{y_k z_i z_{\bj}}
    = \,& \sqrt{-1} ((v_{i\bj k} - v_{i\bj \bk})
                      - (v_{ki\bj} -  v_{\bk i \bj})) \\
    = \,& \sqrt{-1} ((v_{i\bj k} - v_{ki\bj})
                      - (v_{i\bj \bk} - v_{i \bk \bj})) \\
    = \,&  \sqrt{-1} (- g^{l\bm} R_{i \bj k \bm} v_l + T_{ik}^l v_{l\bj}
                       - \ol{T_{jk}^l} v_{i\bl}).
\end{aligned} 
\end{equation}

For convenience we set
\[ 
t_{2k-1} = x_k,\;  t_{2k} = y_k, \, 1 \leq k \leq n-1; \;
  t_{2n-1} = y_n, \; t_{2n} = x_n. \]
By \eqref{gblq-B120}, \eqref{gblq-B130} and the identity
 \begin{equation}
\label{gblq-B170'}
 \fg^{i\bj} T^l_{ki} u_{l\bj} = T^i_{ki} - \fg^{i\bj} T^l_{ki} \chi_{l\bj}
 \end{equation}
we obtain for all $1 \leq \alpha \leq 2n$,
\begin{equation}
\label{gblq-B50}
\begin{aligned}
|\fg^{i\bj}  u_{t_{\alpha} i\bj}|
  \leq \, & |(f)_{t_{\alpha}}| + |\fg^{i\bj} (\chi_{i\bj})_{t_{\alpha}}|
   + |\fg^{i\bj} (u_{t_{\alpha} i\bj} - u_{i\bj t_{\alpha}})|
 \leq  C (1 + \fg^{i\bj} g_{i\bj}).
\end{aligned}
\end{equation}
We also record here the following identity which we shall use later:
for a function $\eta$,
\begin{equation}
\label{gblq-B190}
 \begin{aligned}
  \fg^{i\bj} \eta_i u_{x_n \bj}
   = &\, \fg^{i\bj} \eta_i (2 u_{n \bj} + \sqrt{-1} u_{y_n \bj}) \\
   = &\, 2 \eta_n - 2 \fg^{i\bj} \eta_i \chi_{n \bj}
          + \sqrt{-1} \fg^{i\bj}  \eta_i u_{y_n \bj}.
 \end{aligned}
 \end{equation}

We now start to derive \eqref{cma-37}.
We assume
\begin{equation}
\label{cma-38}
|u| + | \nabla u| \leq K \;\; \mbox{in $\bar{M}$}.
\end{equation}
Set
\begin{equation}
\label{cma-39}
 \ul{\psi} \equiv
       \min_{|u| \leq K, \, z \in \bar{M}} \psi (z, u) > 0, \;\;
   \bar{\psi} \equiv
       \max_{|u| \leq K, \, z \in \bar{M}} \psi (z, u).
\end{equation}

\vspace{.2in}
Let $\sigma$ be the distance function to $\partial M$. Note that
$|\nabla \sigma| = \frac{1}{2}$ on $\partial M$.
 There exists
$\delta_0 > 0$ such that $\sigma$ is smooth and $\nabla \sigma \neq 0$
in
\[ M_{\delta_0} := \{z \in M: \sigma (z) < \delta_0\}, \]
which we call the $\delta_0$-neighborhood of $\partial M$.
We can therefore write
\begin{equation}
\label{gblq-B10}
 u - \ul{u} = h \sigma, \;\; \mbox{in $M_{\delta_0}$}
\end{equation}
where $h$ is a smooth function.

Consider a boundary point $p \in \partial M$. We choose local
coordinates $z = (z_1, \ldots, z_n)$, $z_j = x_j + i y_j$,
around $p$ in a neighborhood which we assume to be contained in
$M_{\delta_0}$ such
$\frac{\partial}{\partial x_n}$ is the interior normal direction to
$\partial M$ at $p$ where we also assume $g_{i\bj} = \delta_{ij}$;
(Here and in what follows we identify $p$ with $z = 0$.) for later
reference we call such local coordinates {\em regular} coordinate
charts.

By \eqref{gblq-B10} we have
\[  (u - \ul{u})_{x_n} = h_{x_n} \sigma + h \sigma_{x_n} \]
and
\[ (u - \ul{u})_{j\bk} =h_{j\bk} \sigma + h \sigma_{j\bk}
                      + 2 \, \fRe \{h_j \sigma_{\bk}\}. \]
Since $\sigma = 0$ on $\partial M$ and
$\sigma_{x_n} (0) = 2 |\nabla \sigma| = 1$, we see that
\[ (u - \ul{u})_{x_n} (0) = h (0) \]
and
\begin{equation}
\label{cma-60'}
(u-\ul{u})_{j\bk}(0) = (u-\ul{u})_{x_n} (0) \sigma_{j\bk}(0)
\;\; j, k < n.
\end{equation}
Similarly,
\begin{equation}
\label{cma-60}
(u-\ul{u})_{t_{\alpha}t_{\beta}} (0) = - (u-\ul{u})_{x_n} (0)
 \sigma_{t_{\alpha} t_{\beta}},
\;\; \alpha, \beta < 2n.
\end{equation}
It follows that
\begin{equation}
|u_{t_{\alpha} t_{\beta}}(0)| \leq C, \;\;\;\; \alpha, \beta < 2n
\label{cma-70}
\end{equation}
where $C$ depends on $|u|_{C^1 (\bM)}$, $|\ul{u}|_{C^1 (\bM)}$, and
the principal curvatures of $\partial M$.

\vspace{.2in}
To estimate $u_{t_{\alpha} x_n} (0)$ for $\alpha \leq 2n$,
we will follow \cite{Guan98b} and employ a barrier function of the form
\begin{equation}
\label{cma-E85}
v = (u - \ul{u}) + t \sigma - N \sigma^2,
\end{equation}
where $t, N$ are positive constants to be determined.
Recall that $\ul{u} \in C^2$ and $\chi_{\ul{u}} > 0$ in a neighborhood of
$\partial M$. We may assume that there exists $\epsilon > 0$ such that
$\chi_{\ul{u}} \geq \epsilon \omega$ in $M_{\delta_0}$. Locally, this gives
\begin{equation}
\label{cma-E86}
\{\ul{u}_{j\bk} + \chi_{j\bk} (\cdot, \ul{u})\} \geq \epsilon \{g_{j\bk}\}.
\end{equation}
The following is the key ingredient in our argument.

\begin{lemma}
\label{cma-lemma-20}
For $N$ sufficiently large and $t, \delta$ sufficiently small,
\[ \begin{aligned}
  \fg^{i\bj} v_{i\bj}
    \leq - \frac{\epsilon}{4} &\, (1 + \sum \fg^{i\bj} g_{i\bj})
   \;\;\; \mbox{in} \;\; \Omega_{\delta}, \\
  v & \, \geq 0 \;\; \mbox{on} \;\; \partial \Omega_{\delta}
  \end{aligned} \]
where $\Omega_{\delta} = M \cap B_{\delta}$ and $B_{\delta}$ is the
(geodesic) ball of radius $\delta$ centered at $p$.
\end{lemma}

\begin{proof}
This lemma was first proved in \cite{Guan98b} for domains in $\bfC^n$.
For completeness we include the proof here with minor modifications.
By \eqref{cma-E86} we have
\begin{equation}
\label{cma-E90}
\fg^{i\bj} (u_{i\bj} -\ul{u}_{i\bj})
   \leq \fg^{i\bj} (u_{i\bj} + \chi_{i\bj} (\cdot, u)
      - \ul{u}_{i\bj} - \chi_{i\bj} (\cdot, \ul{u}))
 \leq n - \epsilon \fg^{i\bj} g_{i\bj}.
\end{equation}
Obviously,
\[ \fg^{i\bj} \sigma_{i\bj} \leq C_1 \fg^{i\bj}  g_{i\bj} \]
for some constant $C_1 > 0$ under control.
Thus
\[  \fg^{i\bj} v_{i\bj} \leq n
    + \{C_1 (t + N \sigma) - \epsilon\} \fg^{i\bj} g_{i\bj}
- 2 N \fg^{i\bj} \sigma_i \sigma_{\bj} \;\;\; \mbox{in} \;\;
\Omega_{\delta}.\]

Let $\lambda_1 \leq \cdots \leq \lambda_n$ be the eigenvalues
of $\{u_{i\bj} + \chi_{i\bj}\}$ (with respect to $\{g_{i\bj}\}$).
We have $\fg^{i\bj}  g_{i\bj} = \sum \lambda_k^{-1}$
and
\begin{equation}
\label{cma-E95}
\fg^{i\bj} \sigma_i \sigma_{\bj} \geq  \frac{1}{2 \lambda_n}
\end{equation}
since $|\nabla \sigma| \equiv \frac{1}{2}$ where $\sigma$ is smooth.
By the arithmetic-geometric mean-value inequality,
\[ \frac{\epsilon}{4} \fg^{i\bj}  g_{i\bj} + \frac{N}{\lambda_n}
  \geq \frac{n \epsilon}{4}
       (N \lambda_1^{-1} \cdots \lambda_n^{-1})^{\frac{1}{n}}
  \geq \frac{n \epsilon  N^{\frac{1}{n}}}{4 \psi^{\frac{1}{n}}}
  \geq c_1 N^{\frac{1}{n}} \]
for some constant $c_1 > 0$ depending on the upper bound of $\psi$.

We now fix $t > 0$ sufficiently small and $N$ large so that
 $c_1 N^{1/n} \geq 1 + n + \epsilon$ and $C_1 t \leq \frac{\epsilon}{4}$.
Consequently,
\[  \fg^{i\bj} v_{i\bj}
    \leq - \frac{\epsilon}{4} (1 + \fg^{i\bj} g_{i\bj})
   \;\;\; \mbox{in} \;\; \Omega_{\delta} \]
if we require $\delta$ to satisfy
$C_1 N \delta \leq \frac{\epsilon}{4}$ in $\Omega_{\delta}$.

On $\partial M \cap B_{\delta}$ we have $v = 0$.
On $M \cap \partial B_{\delta}$,
\[ v \geq t \sigma - N \sigma^2
     \geq  (t - N \delta) \sigma \geq 0 \]
if we require, in addition, $N \delta \leq t$.
\end{proof}

\begin{remark}
For the real Monge-Amp\`ere equations,
Lemma~\ref{cma-lemma-20} was proved in \cite{Guan98a} both for domains
in $\bfR^n$ and in general Riemannian manifolds, improving earlier results
in \cite{HRS}, \cite{GS93} and \cite{GL96}.
\end{remark}

\begin{lemma}
\label{cma-lemma-30}
Let $w \in C^2 (\ol{\Omega_{\delta}})$.
Suppose that $w$ satisfies
\[ \fg^{i\bj} w_{i\bj} \geq - C_1 (1 + \fg^{i\bj} g_{i\bj}) \;\;
\mbox{in $\Omega_{\delta}$} \]
and
\[ w \leq C_0 \rho^2 \;\; \mbox{on $ B_{\delta} \cap \partial M$},
   \;\; w (0) = 0\]
where $\rho$ is the distance function to the point $p$ (where $z = 0$) on
$\partial M$. Then
$ w_{\nu} (0) \leq C$,
where $\nu$ is the interior unit normal to $\partial M$, and
$C$ depends on $\epsilon^{-1}$, $C_0$, $C_1$,
$|w|_{C^0 (\ol{\Omega_{\delta}})}$,
$|u|_{C^1 (\bar{M})}$ and the constants $N$, $t$ and $\delta$ determined
in Lemma~\ref{cma-lemma-20}.
\end{lemma}

\begin{proof}
By Lemma~\ref{cma-lemma-20},
$A v + B \rho^2 - w \geq 0$ on $\partial \Omega_{\delta}$
and
\[ \fg^{i\bj} (A v + B \rho^2 - w)_{i\bj} \leq 0 \;\;
     \mbox{in $\Omega_{\delta}$} \]
when $A \gg B$ and both are sufficiently large. By the maximum principle,
\[ A v + B \rho^2 - w \geq 0 \;\; \mbox{in $\ol{\Omega_{\delta}}$}. \]
Consequently,
\[ A v_{\nu} (0) - w_{\nu} (0) = D_{\nu} (A v + B \rho^2 - w) (0) \geq 0 \]
since $A v + B \rho^2 - w = 0$ at the origin.
\end{proof}

\vspace{.2in}
We next apply Lemma~\ref{cma-lemma-30} to estimate
$u_{t_{\alpha} x_n} (0)$ for $\alpha < 2n$, following \cite{CKNS}.
For fixed $\alpha < 2n$,
we write $\eta = \sigma_{t_{\alpha}}/\sigma_{x_n}$ and
define
\[ \mathcal{T} = \nabla_{\frac{\partial}{\partial t_{\alpha}}}
     - \eta \nabla_{\frac{\partial}{\partial x_n}}. \]
We wish to apply Lemma~\ref{cma-lemma-30}  to
\[ w = (u_{y_n} - \varphi_{y_n})^2 \pm  \cT (u - \varphi). \]

By (\ref{cma-38}),
\[ | \mathcal{T} (u - \varphi)| + (u_{y_n} - \varphi_{y_n})^2 \leq C
    \;\; \mbox{in} \; \Omega_{\delta}. \]
On $\partial M$ since $u - \varphi = 0$ and
$\cT$ is a tangential differential operator, we have
\[  \mathcal{T} (u - \varphi) = 0 \;\; \mbox{on}\;\partial M \cap B_{\delta} \]
and, similarly,
\begin{equation}
\label{cma-175}
(u_{y_n} - \varphi_{y_n})^2 \leq C \rho^2
  \;\; \mbox{on} \; \partial M \cap B_{\delta}.
\end{equation}

We compute next
\begin{equation}
\label{gblq-B160}
\fg^{i\bj}  (\mathcal{T} u)_{i\bj}
   = \fg^{i\bj}  (u_{t_\alpha i\bj} + \eta u_{x_n i\bj})
 + \fg^{i\bj} \eta_{i\bj} u_{x_n} + 2 \fg^{i\bj}  \fRe \{\eta_i u_{x_n \bj}\}.
\end{equation}

By \eqref{gblq-B50} and \eqref{gblq-B190},
\begin{equation}
\label{gblq-B180}
 |\fg^{i\bj}  (u_{t_\alpha i\bj} + \eta u_{x_n i\bj})|
   \leq | \mathcal{T} (f) |
   + C_1 (1 + \fg^{i\bj}  g_{i\bj})
\end{equation}
and
\begin{equation}
\label{gblq-B200}
2 |\fg^{i\bj} \fRe\{\eta_i u_{x_n \bj}\}| \leq \fg^{i\bj} u_{y_n i} u_{y_n \bj}
+ C_2 (1 + \fg^{i\bj}  g_{i\bj}).
\end{equation}
 Applying \eqref{gblq-B50} again, we derive
\begin{equation}
\label{gblq-B210}
\begin{aligned}
\fg^{i\bj} [(u_{y_n} - \varphi_{y_n})^2]_{i\bj}
   = \,& 2 \fg^{i\bj}
         (u_{y_n} - \varphi_{y_n})_i (u_{y_n} - \varphi_{y_n})_{\bj} \\
       & + 2 (u_{y_n} - \varphi_{y_n})
         \fg^{i\bj} (u_{y_n} -  \varphi_{y_n})_{i\bj} \\
\geq \,& \fg^{i\bj} u_{y_n i} u_{y_n \bj}
          - 2 \fg^{i\bj} \varphi_{y_n i} \varphi_{y_n \bj} \\
       & + 2 (u_{y_n} - \varphi_{y_n})
          \fg^{i\bj} (u_{y_n i\bj} - \varphi_{y_n i\bj}) \\
\geq \,& \fg^{i\bj} u_{y_n i} u_{y_n \bj}
          - |(f)_{y_n}| - C_3 (1 + \fg^{i\bj}  g_{i\bj}).
 \end{aligned}
 \end{equation}
Finally, combining \eqref{gblq-B160}-\eqref{gblq-B210} we obtain
\begin{equation}
\label{gblq-B230}
 \fg^{i\bj} [(u_{y_n} - \varphi_{y_n})^2 \pm \mathcal{T} (u - \varphi)]_{i\bj}
  \geq - C_4 (1 + \fg^{i\bj}  g_{i\bj})
  \;\; \mbox{in} \;\Omega_{\delta}.
 \end{equation}

Consequently, we may apply Lemma~\ref{cma-lemma-30} to
$w = (u_{y_n} - \varphi_{y_n})^2 \pm  \cT (u - \varphi)$ to obtain
\begin{equation}
\label{cma-180}
|u_{t_{\alpha} x_n} (0)| 
                         \leq C, \;\;\;\; \alpha < 2n.
\end{equation}
By \eqref{gblq-B110} we also have
\begin{equation}
\label{cma-180'}
|u_{x_n t_{\alpha}} (0)| 
                         \leq C, \;\; \alpha < 2n.
\end{equation}

\vspace{.2in}
It remains to establish the estimate
\begin{equation}
|u_{x_n x_n} (0)| \leq  C.
\end{equation}
Since we have already derived
\begin{equation}
\label{cma-190}
|u_{t_{\alpha} t_{\beta}} (0)|, \; |u_{t_{\alpha} x_n} (0)|, \;
|u_{x_n t_{\alpha}} (0)| \leq C,
\;\;\;\; \alpha, \beta < 2n,
\end{equation}
it suffices to prove
\begin{equation}
\label{cma-200}
0 \leq \chi_{n \bn} (0) + u_{n \bn} (0)
  = \chi_{n \bn} (0) + u_{x_n x_n} (0) + u_{y_n y_n} (0) \leq C.
\end{equation}

Expanding $\det (u_{i\bj} + \chi_{i\bj})$, we have
\begin{equation}
\label{gblq-B310}
 \det (u_{i\bj} (0)+ \chi_{i\bj} (0))
    = a (u_{n\bn} (0) + \chi_{n \bn} (0)) + b
\end{equation}
where
\[ a = \det (u_{\alpha \bar{\beta}} (0) + \chi_{\alpha \bar{\beta}} (0))
   |_{\{1 \leq \alpha, \beta \leq n-1\}} \]
and $b$ is bounded in view of \eqref{cma-190}.
Since $\det (u_{i\bj} + \chi_{i\bj})$ is bounded,
 we only have to derive an {\em a priori} positive lower bound for $a$, which
is equivalent to
\begin{equation}
\label{cma-80}
\sum_{\alpha, \beta < n}
  (u_{{\alpha} \bar{\beta}} (0) + \chi_{\alpha \bar{\beta}} (0))
   \xi_{\alpha} \bar{\xi}_{\beta} \geq c_0 |\xi|^2, \;\;
\;\; \forall \, \xi 
\in {\bfC}^{n-1}
\end{equation}
for a uniform constant $ c_0 > 0$.

\begin{proposition}
\label{prop-cma-10}
 There exists $c_0 = c_0 (\ul{\psi}^{-1},
\varphi, \ul{u}) > 0$ such that (\ref{cma-80})
holds.
\end{proposition}

\begin{proof}
Let $T_C \partial M \subset T_C M $
be the complex tangent bundle of $\partial M$ and
\[ T^{1,0} \partial M =  T^{1,0} M \cap T_C \partial M
   = \Big\{\xi \in T^{1,0} M: d \sigma (\xi) = 0\Big\}. \]
 In local coordinates,
\[ T^{1,0} \partial M
   = \Big\{\xi = \xi_i \frac{\partial}{\partial z_i} \in T^{1,0} M:
        \sum \xi_i \sigma_i = 0 \Big\}. \]
It is enough to establish a positive lower bound for
\[ m_0 = 
       \min_{\xi \in T^{1,0} \partial M,  |\xi| = 1}
      \chi_u (\xi, \bar{\xi}). \]

We assume that $m_0$ is attained at a point $p \in \partial M$ and
choose regular local coordinates 
around $p$
as before such that
\[ m_0 = \chi_u \Big(\frac{\partial}{\partial z_1},
    \frac{\partial}{\partial \bz_1}\Big)
       = u_{1\bar{1}} (0) + \chi_{1\bar{1}} (0). \]
One needs to show
\begin{equation}
\label{cma-90}
m_0 = u_{1\bar{1}} (0) + \chi_{1\bar{1}} (0) \geq c_0 > 0.
\end{equation}
By \eqref{cma-60'},
\begin{equation}
\label{cma-100}
u_{1\bar{1}} (0) = \ul{u}_{1\bar{1}} (0)
        - (u-\ul{u})_{x_n} (0) \sigma_{1\bar{1}} (0).
\end{equation}
We can assume
$u_{1\bar{1}} (0)
 \leq \frac{1}{2} (\ul{u}_{1\bar{1}} (0) - \chi_{1\bar{1}} (0)$;
otherwise we are done.
Thus 
\begin{equation}
\label{cma-100'}
(u-\ul{u})_{x_n} (0) \sigma_{1\bar{1}} (0)
  \geq \frac{1}{2} (\ul{u}_{1\bar{1}} (0) + \chi_{1\bar{1}} (0)).
\end{equation}
It follows from \eqref{cma-38} that
\begin{equation}
\label{cma-101}
\sigma_{1\bar{1}} (0) \geq
  \frac{\ul{u}_{1\bar{1}} (0) + \chi_{1\bar{1}} (0)}{2 K}
                      \geq \frac{\epsilon}{2 K} \equiv c_1 > 0
\end{equation}
where $K = \max_{\partial M} |\nabla (u-\ul{u})|$.

Let $\delta > 0$ be small enough so that
\[ \begin{aligned}
  w \equiv \, &  \Big|- \sigma_{z_n} \frac{\partial}{\partial z_1}
   + \sigma_{z_1} \frac{\partial}{\partial z_n}\Big|  \\
 = \, & \left(g_{1\bar{1}} |\sigma_{z_n}|^2 - 2 \fRe \{g_{1\bar{n}} \sigma_{z_n} \sigma_{\bz_1}\}
   + g_{n\bar{n}} |\sigma_{z_1}|^2\right)^{\frac{1}{2}} > 0 \;\;
\mbox{in $M \cap B_{\delta} (p)$}.
\end{aligned} \]
Define $\zeta = \sum \zeta_i \frac{\partial}{\partial z_i} \in T^{1,0} M$
in $M \cap B_{\delta} (p)$:
\[ \left\{ \begin{aligned}
   \zeta_1 & \, = - \frac{\sigma_{z_n}}{w}, \\
   \zeta_j & \, = 0, \; \; 2 \leq j \leq n-1, \\
   \zeta_n & \, = \frac{\sigma_{z_1}}{w}
   \end{aligned} \right.  \]
and
\[ \varPhi = (\varphi_{j\bk} + \chi_{j\bk}) \zeta_j \bar{\zeta}_k
   - (u - \varphi)_{x_n} \sigma_{j\bk} \zeta_j \bar{\zeta}_k
   - u_{1\bar{1}}(0) - \chi_{1\bar{1}} (0). \]
Note that $\zeta \in T^{1,0} \partial M$ on $\partial M$ and $|\zeta| = 1$.
By \eqref{cma-60'},
\begin{equation}
\label{cma-103}
\varPhi = (u_{j\bk} + \chi_{j\bk}) \zeta_j \bar{\zeta}_k
  - u_{1\bar{1}}(0) - \chi_{1\bar{1}} (0) \geq 0
\;\;  \mbox{on $\partial M \cap B_{\delta} (p)$}
\end{equation}
and $ \varPhi (0) = 0$.

Write $G = \sigma _{i\bj} \zeta_i \bar{\zeta}_i$. We have
\begin{equation}
\label{gblq-B360}
 \begin{aligned}
\fg^{i\bj} \varPhi_{i\bj}
 \leq \, & - \fg^{i\bj} (u_{x_n} G)_{i\bj} + C (1 + \fg^{i\bj}  g_{i\bj}) \\
    = \, & - G \fg^{i\bj} u_{x_n i\bj} - 2 \fg^{i\bj} \fRe\{u_{x_n i} G_{\bj}\}
        + C (1 + \fg^{i\bj}  g_{i\bj}) \\
 \leq \, & \fg^{i\bj} u_{y_n i} u_{y_n \bj}
       + C (1 + \fg^{i\bj}  g_{i\bj})
\end{aligned}
\end{equation}
by \eqref{gblq-B50} and \eqref{gblq-B190}.
It follows that
\begin{equation}
\label{cma-105}
\fg^{i\bj} [\varPhi - (u_{y_n} - \varphi_{y_n})^2]_{i\bj}
      \leq C (1 + \fg^{i\bj}  g_{i\bj})
        \;\; \mbox{in $M \cap B_{\delta} (p)$}.
\end{equation}
Moreover, by \eqref{cma-175} and \eqref{cma-103},
\[ (u_{y_n} - \varphi_{y_n})^2 - \varPhi  \leq C |z|^2 \;\;
 \mbox{on $\partial M \cap B_{\delta} (p)$}. \]
Consequently, we may apply Lemma~\ref{cma-lemma-30} to
\[ h = (u_{y_n} - \varphi_{y_n})^2 - \varPhi \]
to derive $\varPhi_{x_n} (0) \geq -C$
which, by \eqref{cma-101}, implies
\begin{equation}
\label{cma-310}
 u_{x_n x_n} (0) \leq \frac{C}{\sigma_{1\bar{1}} (0)} \leq \frac{C}{c_1}.
\end{equation}

In view of \eqref{cma-190} and \eqref{cma-310} we have an {\em a priori}
upper bound for all eigenvalues of $\{u_{i\bj} (0) + \chi_{i\bj} (0)\}$.
Since $\det (u_{i\bj} + \chi_{i\bj}) \geq \ul{\psi} > 0$, the eigenvalues of
$\{u_{i\bj} (0)  + \chi_{i\bj} (0) \}$ must admit a positive lower bound,
i.e.,
\[ \min_{\xi \in T^{1,0}_p M, |\xi| = 1}
    (u_{i\bj} + \chi_{i\bj}) \xi_i \bar{\xi}_j \geq c_0.\]
Therefore,
\[ m_0 = \min_{\xi \in T_p^{1,0} \partial M, |\xi| = 1}
         (u_{i\bj} + \chi_{i\bj}) \xi_i \bar{\xi}_j
   \geq \min_{\xi \in T^{1,0}_p M, |\xi| = 1}
         (u_{i\bj} + \chi_{i\bj}) \xi_i \bar{\xi_j}
    \geq c_0.\]
The proof of Proposition~\ref{prop-cma-10} is complete.
\end{proof}

We have therefore established (\ref{cma-37}).

\bigskip

\section{Estimates for the real Hessian and higher derivatives}
\label{gblq-R}
\setcounter{equation}{0}
\medskip

The primary goal of this section is to derive global estimates for the whole
(real) Hessian
\begin{equation}
\label{gblq-R5}
|\nabla^2 u| \leq  C  \;\; \mbox{on $\bM$}.
\end{equation}
This is equivalent to
\begin{equation}
\label{cma-410'}
 |u_{x_i x_j} (p)|, \; |u_{x_i y_j} (p)|, \;  |u_{y_i y_j} (p)| \leq C, \;\;
 \forall \, 1 \leq i, j \leq n
\end{equation}
in local coordinates  $z = (z_1, \ldots, z_n)$,
$z_j = x_j + \sqrt{-1} y_j$ with $g_{i\bj} (p) = \delta_{ij}$ for
any fixed point $p \in M$,
where the constant $C$ may depend on $|u|_{C^1 (M)}$,
$\sup_M \Delta u$, $\inf \psi > 0$, and the curvature and torsion of $M$ as
well as their derivatives. Once this is done we can apply the Evans-Krylov
Theorem to obtain global $C^{2, \alpha}$ estimates.

As in Section~\ref{gblq-B} we shall use covariant derivatives.
We start with communication formulas for the fourth order derivatives.
From  direct computation,
\begin{equation}
\label{gblq-R145}
\left\{
\begin{aligned}
v_{ij \bk \bl} - v_{ij \bl \bk} = \, & \ol{T_{kl}^q} v_{ij \bq}, \\
v_{ij k \bl} - v_{ij \bl k} = \, & g^{p\bq} R_{k\bl i\bq} v_{pj}
                                 + g^{p\bq} R_{k\bl j\bq} v_{ip}.
\end{aligned}
 \right.
\end{equation}
Therefore, by \eqref{gblq-B145}, \eqref{gblq-B150}, \eqref{gblq-B155},
\eqref{gblq-R145} and \eqref{cma-K115},
\begin{equation}
\label{gblq-R150}
 \begin{aligned}
v_{i \bj k \bl} - v_{k \bl i \bj}
 = \,& (v_{i \bj k \bl} - v_{k i \bj \bl})
         + (v_{k i \bj \bl} - v_{k i \bl \bj})
         + (v_{k i \bl \bj} - v_{k \bl i \bj}) \\
 = \,& \nabla_{\bl} (- g^{p\bq} R_{i \bj k \bq} v_p + T_{ik}^p v_{p\bj})
     + \ol{T_{jl}^q} v_{ki\bq} + g^{p\bq} \nabla_{\bj} (R_{i \bl k \bq} v_p)\\
 = \,& T_{ik}^p v_{p\bj \bl} + \ol{T_{jl}^q} v_{i\bq k}
       - T_{ik}^p \ol{T_{jl}^q} v_{p\bq}
       + g^{p\bq} (R_{k\bl i\bq} v_{p\bj} - R_{i\bj k\bq} v_{p\bl}) \\
   & + g^{p\bq} (\nabla_{\bj} R_{i\bl k\bq} - \nabla_{\bl} R_{i\bj k\bq}
              + R_{i\bm k\bq} \ol{T_{jl}^m}) v_p
\end{aligned}
\end{equation}
and
\begin{equation}
\label{gblq-R155}
 \begin{aligned}
v_{i \bj k l} - v_{k l i \bj}
   = \,& v_{i \bj k l} - v_{k i \bj l}
         + v_{k i \bj l} - v_{k i l \bj}
         + v_{k i l \bj} - v_{k l i \bj} \\
   = \,& \nabla_{l} (- g^{p\bq} R_{i \bj k \bq} v_p + T_{ik}^p v_{p\bj})
     -  g^{p\bq} R_{l\bj k\bq} v_{pi} - g^{p\bq} R_{l\bj i\bq} v_{kp} \\
     & + \nabla_{\bj}  (g^{p\bq} R_{ilk\bq} v_p + T_{il}^p v_{kp})\\
   = \,& -g^{p \bq}R_{i\bj k \bq}v_{pl}-g^{p \bq}R_{i \bj l \bq}v_{kp}-g^{p \bq}R_{l \bj k \bq}v_{pi} \\
   & -g^{p \bq}[(\nabla_l R_{i \bj k \bq})+(\nabla_{\bj} R_{ilk \bq})]v_p\\
   & + [(\nabla_l T_{ik}^p)+g^{p \bq} R_{i l k \bq}]v_{p \bj}\\
   & + T_{ik}^p v_{p \bj l}+T_{il}^p v_{k p \bj}.
\end{aligned}
\end{equation}

We now turn to the proof of \eqref{cma-410'}. It suffices to derive
the following estimate.

\begin{proposition}
\label{gblq-prop-R10}
 There exists constant $C > 0$ depending on $|u|_{C^1 (\bM)}$,
$\sup_M \Delta u$ and $\inf \psi > 0$ such that
\begin{equation}
\label{gblq-R5'}
 \sup_{\tau \in T M, |\tau| = 1} u_{\tau \tau}  \leq C.
\end{equation}
\end{proposition}

\begin{proof}
Suppose that
\[ N :=\sup_M \Big\{|\nabla u|^2 + A |\chi_u|^2
       + \sup_{\tau \in T M, |\tau| = 1} u_{\tau \tau} \Big\} \]
is achieved at an interior point $p \in M$ and for some
unit vector $\tau \in T_p M$, where $A$ is positive constant to be determined.
We choose local coordinates $z = (z_1, \ldots, z_n)$
such that $g_{i\bj} = \delta_{ij}$ and $\{u_{i\bj} + \chi_{i\bj}\}$ is diagonal
at $p$. Thus $\tau$ can be written in the form
\[ \tau = a_j \frac{\partial}{\partial z_j}
          + b_{\bj} \frac{\partial}{\partial \bz_j},
\;\; a_j, b_{\bj} \in \bfC, \;\;  \sum a_j b_{\bj} = \frac{1}{2}.  \]

Let $\xi$ be a smooth unit vector field defined in a neighborhood of $p$
such that $\xi (p) = \tau$. Then the function
\[ Q = u_{\xi \xi} + |\nabla u|^2 + A |\chi_u|^2 \]
(defined in a neighborhood of $p$) attains it maximum at $p$ where,
\begin{equation}
\label{gblq-R20}
Q_i = u_{\tau \tau i} + u_k u_{i\bk} + u_{\bk} u_{ki}
   + 2 A (u_{k\bk}+ \chi_{k\bk}) (u_{k\bk} + \chi_{k\bk})_i = 0
\end{equation}
and
\begin{equation}
\label{gblq-R30}
\begin{aligned}
0 \geq \fg^{i\bi} Q_{i\bi}
 =  \,& \fg^{i\bi}(u_{k\bi} u_{i\bk} + u_{ki} u_{\bk \bi})
        +  \fg^{i\bi} (u_k u_{i\bk \bi} + u_{\bk} u_{ki\bi}) \\
      + \fg^{i\bi} \,& u_{\tau \tau i\bi}
   + 2 A \fg^{i\bi} (u_{k\bl}+\chi_{k\bl})_i (u_{l\bk}+(\chi_{l\bk})_{\bi} \\
   &  + 2 A (u_{k\bk}+\chi_{k\bk}) \fg^{i\bi}
              (u_{k\bk i\bi}+(\chi_{k\bk})_{i\bi}).
\end{aligned}
\end{equation}

Differentiating equation~(\ref{cma2-M10}) twice (using covariant derivatives),
by \eqref{gblq-B150} and \eqref{gblq-R150} we have
\begin{equation}
\label{gblq-R40}
 \fg^{i\bi} u_{ki\bi} = (f)_k
  + \fg^{i\bi} (R_{i\bi k\bl} u_l - T_{ik}^l u_{l\bi} - (\chi_{i\bi})_k)
                   \geq (f)_k - C \Big(1 + \sum \fg^{i\bi}\Big)
\end{equation}
and
\begin{equation}
\label{gblq-R50}
\begin{aligned}
\fg^{i\bi} u_{k\bk i\bi}
 \geq  \fg^{i\bi} \fg^{j\bj} |u_{i\bj k} \,& + (\chi_{i\bj})_k|^2
    + \fg^{i\bi}  (T_{ik}^p u_{p \bi \bk} + \ol{T_{ik}^p} u_{i\bp k}) \\
     & + (f)_{k\bk} - C \sum \fg^{i\bi}
 \geq (f)_{k\bk} - C \Big(1 + \sum \fg^{i\bi}\Big).
\end{aligned}
\end{equation}

Note that
\[ u_{\tau \tau i\bi} = a_k a_l u_{kl i\bi}
   + 2 a_k b_{\bl}  u_{k\bl i\bi} + b_{\bk} b_{\bl} u_{\bk\bl i\bi}. \]
Using the formulas in
\eqref{gblq-R145}, \eqref{gblq-R150} and \eqref{gblq-R155} we obtain
\begin{equation}
\label{gblq-R60}
 \begin{aligned}
 \fg^{i\bi} u_{\tau \tau i\bi}
 \geq \,&  \fg^{i\bi} u_{i\bi \tau \tau} - C \fg^{i\bi} |T_{ik}^l u_{l\bi k}|
-  C \Big(1 + \sum_{k, l} |u_{kl}|\Big) \sum \fg^{i\bi} \\
 \geq \,& (f)_{\tau \tau} - C \fg^{i\bi} u_{l\bi k} u_{i\bl k}
-  C \Big(1 + \sum_{k, l} |u_{kl}|\Big) \sum \fg^{i\bi}.
\end{aligned}
\end{equation}
From \eqref{gblq-R40}, \eqref{gblq-R50}, \eqref{gblq-R60},
\eqref{gblq-R30} and the inequality
\begin{equation}
\label{gblq-R70}
\begin{aligned}
 2 \fg^{i\bi}
  (u_{k\bl i} + \chi_{k\bl})_i (u_{l\bk \bi}+ \chi_{l\bk})_{\bi}
 \geq \,& \fg^{i\bi} u_{k\bl i} u_{l\bk \bi}
     - \fg^{i\bi} (\chi_{k\bl})_i (\chi_{l\bk})_{\bi},
\end{aligned}
\end{equation}
we see that
\begin{equation}
\label{gblq-R80}
\fg^{i\bi} u_{ki} u_{\bk \bi} + (A-C) \fg^{i\bi} u_{k\bl i} u_{l\bk \bi}
\leq C \Big(1 + A +  \sum_{k, l} |u_{kl}| \Big) \Big(1 + \sum \fg^{i\bi}\Big).
\end{equation}
We now need the nondegeneracy of equation~\eqref{cma2-M10} which implies that
there is $\Lambda > 0$ depending on $\sup_M \Delta u$ and $\inf \psi > 0$
such that
\[ \Lambda^{-1} \{g_{i\bj}\} \leq \{\fg_{i\bj}\} \leq \Lambda \{g_{i\bj}\} \]
and therefore,
\begin{equation}
\label{gblq-R90}
 \sum \fg^{i\bi} \leq n \Lambda, \;\;
 \fg^{i\bi} u_{ki} u_{\bk \bi}
    \geq  \frac{1}{\Lambda} \sum_{i, k} | u_{ki}|^2.
\end{equation}
Plugging these into \eqref{gblq-R80} and choosing $A$ large we derive
\[ \sum_{i, k} | u_{ki}|^2 \leq C.  \]
Consequently, $u_{\tau \tau} (p) \leq C$.
Finally, 
\[  \sup_{q \in M} \sup_{\tau \in T_q M, |\tau| = 1} u_{\tau \tau}
   \leq u_{\tau \tau} (p) + 2 \sup_M (|\nabla u|^2 + A |\chi_u|^2). \]
This completes the proof of \eqref{gblq-R5'}.
\end{proof}

By the Evans-Krylov Theorem (\cite{Evans}, \cite{Krylov82},
\cite{Krylov83}) we derive the $C^{2, \alpha}$ estimates
\begin{equation}
\label{gblq-R100}
|u|_{C^{2, \alpha} (M)} \leq C.
\end{equation}
Higher order regularity and estimates then follow from the classical
Schauder theory for elliptic linear equations.

\begin{remark}
\label{gblq-remark-R10}
When $M$ is a K\"ahler manifold, Proposition~\ref{gblq-prop-R10} was recently
proved by Blocki~\cite{Blocki}. He observed that the estimate \eqref{gblq-R5'}
does not depend on $\inf \psi$ when $M$ has nonnegative bisectional curvature.
This is clearly also true in the Hermitian case.
\end{remark}

\begin{remark}
\label{gblq-remark-R20}
An alternative approach to the $C^{2, \alpha}$ estimate \eqref{gblq-R100} is
to use \eqref{gblq-I60} and the boundary estimate \eqref{cma-37} (in place of
\eqref{gblq-R5}) and apply an extension of the Evans-Krylov Theorem; see
Theorem 7.3, page 126 in \cite{CW} which only requires $C^{1, \alpha}$ bounds
for the solution. This was pointed out to us by Pengfei Guan to whom we wish
to express our gratitude.
\end{remark}

\bigskip

\section{$C^0$ estimates and existence}
\label{gblq-E}
\setcounter{equation}{0}
\medskip

In this section we complete the proof of
Theorems~\ref{gblq-th20} and \ref{gblq-th30}-\ref{gblq-th50}
using the estimates established in previous sections.
We shall consider separately the Dirichlet problem and
the case of manifolds without boundary. In each case we need first to
derive $C^0$ estimates; the existence of solutions then can be proved
by the continuity method or combined with degree arguments.

\subsection{Compact manifolds without boundary}
For the $C^0$ estimate on compact manifolds without boundary, we follow the
argument in \cite{Siu87}, \cite{Tian00} which simplifies the original proof
of Yau~\cite{Yau78}.

Let $(M, \omega)$ be a compact Hermitian manifold without boundary.
Replacing $\chi$ by $\chi_{\phi}$ for $\phi \in \cH_{\chi} \cap C^{\infty}(M)$
if necessary,
we shall assume $\chi \geq \epsilon \omega$. Let $u \in C^{4}(M)$ be
 an admissible solution of equation~\eqref{cma2-M10},  $\sup_M u=-1$.
We write
\[ \tilde{\chi} = \sum_{k=1}^{n} \chi^{k-1} \wedge (\chi_u)^{n-k}. \]
Multiply the identity
$(\chi_u)^n -\chi^n = \frac{\sqrt{-1}}{2} \p \bar \p u \wedge \tilde{\chi}$
 by $(-u)^p$ and integrate over $M$,
\begin{equation}
\label{gblq-E120}
\begin{aligned}
 \int_M (-u)^p \,& [(\chi_u)^n -\chi^n]
   = \frac{\sqrt{-1}}{2} \int_M (-u)^p \p \bar \p u \wedge \tilde{\chi} \\
   = \,& \frac{p \sqrt{-1}}{2}
          \int_M (-u)^{p-1} \p u \wedge \bar \p u \wedge \tilde{\chi}
       + \frac{\sqrt{-1}}{2} \int_M (-u)^p \bar \p u \wedge \p \tilde{\chi} \\
   = \,& \frac{2p \sqrt{-1}}{(p+1)^2} \int_M
     \p (-u)^{\frac{p+1}{2}} \bar\p (-u)^{\frac{p+1}{2}} \wedge \tilde{\chi}
       - \frac{\sqrt{-1}}{2(p+1)} \int_M (-u)^{p+1} \p \bar \p \tilde{\chi}.
 \end{aligned}
\end{equation}
We now assume that $\p \bar\p \chi^k = 0$, for $k =1, 2$,
 which implies $\p \bar\p \tilde{\chi} = 0$, and
that $\psi$ does not depend on $u$.
Since $\chi > 0$ and $\chi_u \geq 0$,
we see that $\chi^{k-1} \wedge (\chi_u)^{n-k} \geq 0$ for all $k$.
Therefore,
\begin{equation}
\label{gblq-E130}
\begin{aligned}
    \epsilon^n \int_M  |\nabla (-u)^{\frac{p+1}{2}}|^2 \omega^n
  \leq  \,&  \int_M  |\nabla (-u)^{\frac{p+1}{2}}|^2 \chi^n \\
   = \,& 
       \frac{\sqrt{-1}}{2}
       \int_M \p (-u)^{\frac{p+1}{2}} \bar\p (-u)^{\frac{p+1}{2}}
               \wedge \chi^{n-1} \\
\leq \,& \frac{\sqrt{-1}}{2}
         \int_M \p (-u)^{\frac{p+1}{2}} \bar\p (-u)^{\frac{p+1}{2}}
               \wedge \tilde{\chi} \\
   = \,& \frac{(p+1)^2}{2 p} \int_M (-u)^p (\psi \omega^n - \chi^n) \\
 \leq \,& C \int_M (-u)^{p+1} \o^n.
 \end{aligned}
\end{equation}
After this we can derive a bound for $\inf u$ by
the Moser iteration method, following the argument in \cite{Tian00}.

\begin{remark}
\label{gblq-remark-E30}
For $n=2$ we have $\p \bar \p \tilde{\chi} = 2 \p \bar \p \chi$
so the last term in \eqref{gblq-E120} is bounded by
$C \int_M (-u)^{p+1} \omega^n$. Thus the $C^0$ bound
holds for $n=2$ without assumption~\eqref{gblq-I70}.
\end{remark}

If $\psi$ depends on $u$ and satisfies \eqref{gblq-I90}, a bound for
$\sup_M |u|$ follows directly from equation~\eqref{cma2-M10} by the maximum
principle.
Indeed, suppose $u (p) = \max_M u$ for some $p \in M$.
Then $\{u_{i\bj} (p)\} \leq 0$ and, therefore
\[ \det \chi_{i\bj} \geq  \det (u_{i\bj} + \chi_{ij})
= \psi (p, u(p)) \det g_{i\bj}. \]
This implies an upper bound $u (p) \leq C$ by \eqref{gblq-I90}.
That $\min_M u \geq -C$ follows from a similar argument.

\begin{proof}[Proof of Theorem~\ref{gblq-th30}]
We first consider the case that $\psi$ does not depend on $u$.
By assumption~\eqref{gblq-I70} we see that
\[ \int_M (\chi_u)^n = \int_M \chi^n. \]
Therefore,
\[ \int_M \chi^n = \int_M \psi \omega^n  \]
is a necessary condition for the existence of admissible solutions,
and that the linearized operator,
$v \mapsto \fg^{i\bj} v_{i\bj}$, 
of equation~\eqref{gblq-I10} is self-adjoint (with respect to the volume
form $(\chi_u)^n$).
So the continuity method proof in \cite{Yau78} works to give a unique
admissible solution $u \in \cH \cap C^{2, \alpha} (M)$ of \eqref{gblq-I10}
satisfying
\[ \int_M u \omega^n = 0. \]
The smoothness of $u$ follows from the Schauder regularity theory.

For the general case under the assumption $\psi_u \geq 0$, one
can still follow the proof of Yau~\cite{Yau78}. So we omit it here.
\end{proof}

\begin{proof}[Proof of Theorem~\ref{gblq-th40}]
The uniqueness follows easily from the assumption $\psi_u > 0$ and the
maximum principle. For the existence we make use of the continuity method.
For $0 \leq s \leq 1$ consider
\begin{equation}
\label{gblq-E10}
   (\chi_u)^n
  = \psi^s (z, u) \omega^n \;\; \mbox{in $M$}
\end{equation}
where $\psi^s (z, u) = (1-s) e^u + s \psi (z, u)$. Set
\[ S := \{s \in [0,1]: \mbox{equation~\eqref{gblq-E10} is solvable in
           $\cH_{\chi} \cap C^{2, \alpha} (M)$}\} \]
and let $u^s \in \cH_{\chi} \cap C^{2, \alpha} (M)$ be the unique solution of
\eqref{gblq-E10} for $s \in S$. Obviously $S \neq \emptyset$ as
$0 \in S$ with $u^0 = 0$. Moreover, by the $C^{2, \alpha}$ estimates
we see that $S$ is closed. We need to show that $S$ is also open in
and therefore equal to $[0, 1]$; $u^1$ is then the desired solution.

Let $s \in S$ and let $\Delta^s$ denote the Laplace operator of
$(M, \chi_{u^s})$. In local coordinates,
\[ \Delta^s v = \fg^{i\bj} v_{i\bj} = \fg^{i\bj}  \partial_i \bpartial_j v \]
where $\{{\fg}^{i\bj}\} = \{\fg^{s}_{i\bj}\}^{-1}$ and
$\fg^{s}_{i\bj} = \chi_{i\bj} + u^s_{i\bj}$.
Note that $\Delta^s - \psi^s_u$,
where $\psi^{s}_u = \psi^{s}_u (\cdot, u^{s})$,
is the linearized operator of equation~\eqref{gblq-E10} at $u^s$, .
We wish to prove that  for any $\phi \in C^{\alpha} (M, \chi_{u^s})$
there exists a unique solution $v \in C^{2, \alpha} (M, \chi_{u^s})$
to the equation
\begin{equation}
\label{gblq-E30}
\Delta^s v - \psi^s_u v = \phi,
\end{equation}
which implies by the implicit function theorem that
$S$ contains a neighborhood of $s$ and hence is open in $[0,1]$, completing
the proof.

The proof follows a standard approach, using the Lax-Milgram theorem and
the Fredholm alternative. For completeness we include it here.

Let $\gamma > 0$ and define a bilinear form on the Sobolev space
$H^1 (M, \chi_{u^s})$ by
\begin{equation}
\label{gblq-E40}
\begin{aligned}
B [v, w] := \,& \int_M [\langle \nabla v + v \mbox{tr} \tilde{T},
                                  \nabla w \rangle_{\chi_{u^s}}
                  + (\gamma + \psi^{s}_u) v w] (\chi_{u^{s}})^n \\
          = \,& \int_M [{\fg}^{i\bj} (v_i + v \tilde{T}_{ik}^k)w_{\bj}
               + (\gamma + \psi^{s}_u) v w] (\chi_{u^{s}})^n
\end{aligned}
\end{equation}
 where $\tilde{T}$ denotes the torsion of $\chi_{u^s}$ and
$\mbox{tr} \tilde{T}$ its trace. In local coordinates,
\[ \mbox{tr} \tilde{T} = \tilde{T}_{ik}^k d z_i
      = {\fg}^{k\bj} (\chi_{i\bj k} - \chi_{k\bj i}) dz_i \]
so it only depends on the second derivatives of $u$.

It is clear that for $\gamma > 0$ sufficiently large $B$ satisfies the
Lax-Milgram hypotheses, i.e,
\begin{equation}
\label{gblq-E50}
|B [v, w]| \leq C \|v\|_{H^1} \|w\|_{H^1}
\end{equation}
by the Schwarz inequality, and
\begin{equation}
\label{gblq-E60}
 B [v, v] \geq c_0 \|v\|_{H^1}^2,
   \;\; \forall \; v \in H^1 (M, \chi_{u^s})
\end{equation}
where $c_0$ is a positive constant independent of $s \in [0,1]$ since
$\psi_u > 0$, $|u^s|_{C^2 (M)} \leq C$ and $M$ is compact.
By the Lax-Milgram theorem, for any $\phi \in L^2 (M, \chi_{u^s})$
there is a unique $v \in H^1 (M, \chi_{u^s})$ satisfying
\begin{equation}
\label{gblq-E70}
B  [v, w] = \int_M  \phi w (\chi_{u^{s}})^n \;\;
    \forall \; w \in H^1 (M, \chi_{u^s}).
\end{equation}
On the other hand,
\begin{equation}
\label{gblq-E80}
    B [v, w] = \int_M (- \Delta^s v + \psi^s_u v + \gamma v) w
               (\chi_{u^{s}})^n
\end{equation}
by integration by parts. Thus $v$ is a weak solution to the equation
\begin{equation}
\label{gblq-E90}
 L_{\gamma} v := \Delta^s v - \psi^s_u v - \gamma v = \phi.
\end{equation}
We write $v = L_{\gamma}^{-1} \phi$.

By the Sobolev embedding theorem the linear operator
\[ K := \gamma L_{\gamma}^{-1}: L^2 (M, \chi_{u^s}) \rightarrow L^2 (M, \chi_{u^s}) \]
is compact. Note also that $v \in H^1 (M, \chi_{u^s})$ is a weak solution of
equation~\eqref{gblq-E30} if and only if
\begin{equation}
\label{gblq-E100}
 v - K v = \zeta
\end{equation}
where $\zeta = L_{\gamma}^{-1} \phi$. Indeed, \eqref{gblq-E30} is
equivalent to
\begin{equation}
\label{gblq-E105}
v = L_{\gamma}^{-1} (\gamma v + \phi) = \gamma L_{\gamma}^{-1} v
   + L_{\gamma}^{-1} \phi.
\end{equation}
Since the solution of equation~\eqref{gblq-E30}, if exists, is unique,
by the Fredholm alternative equation~\eqref{gblq-E100} is uniquely solvable
for any $\zeta \in L^2 (M, \chi_{u^s})$.
Consequently, for any $\phi \in L^2 (M, \chi_{u^s})$ there exists a unique
solution $v \in H^1 (M, \chi_{u^s})$ to equation~\eqref{gblq-E30}.
By the regularity theory of linear elliptic equations,  $v \in C^{2, \alpha} (M, \chi_{u^s})$
if  $\phi \in C^{\alpha} (M, \chi_{u^s})$.
This completes the proof.
\end{proof}

\subsection{The Dirichlet problem}
We now turn to the proof of Theorem~\ref{gblq-th20}. Let
\[ \cA_{\ul{u}} = \{v \in \cH_{\chi}: \mbox{$v \geq \ul{u}$ in $M$,
                    $v = \ul{u}$ on $\partial M$}\}. \]
By the maximum principle, $v \leq h$ on $\bM$ for all $v \in \cA_{\ul{u}}$
where $h$ satisfies $\Delta h + \tr \chi = 0$ in $M$ and $h = \ul{u}$ on
$\partial M$. Therefore we have $C^0$ bounds for solutions
of the Dirichlet problem~\eqref{gblq-I10}-\eqref{gblq-I20} in $\cA_{\ul{u}}$.
The proof of existence of such solutions then follows that of Theorem~1.1 in
\cite{Guan98a}; so is omitted here.

\begin{proof}[Proof of Theorem~\ref{gblq-th50}]
As we only assume $\psi \geq 0$, equation~\eqref{gblq-I10} is
degenerate. So we need to approximate it by nondegenerate equations.

For $\varepsilon > 0$, let $\psi^{\varepsilon}$ be a
smooth function such that
\[ \sup \Big\{\psi - \varepsilon, \frac{\varepsilon^n}{2}\Big\}
\leq \psi^{\varepsilon} \leq \sup \{\psi, \varepsilon^n \} \]
and consider the approximation problem
 \begin{equation}
\label{cma-K700'}
 \left\{ \begin{aligned}
 & (\chi_u)^n = \psi^\varepsilon \omega^{n} \;\; \mbox{in $\bM$}, \\
 & u = \varphi \;\;  \mbox{on $\partial M$}.
  \end{aligned} \right.
\end{equation}
Note that $\ul{u}$ is a subsolution of \eqref{cma-K700'} when
$0 < \varepsilon \leq \epsilon$ where $\epsilon > 0$ satisfies
$\chi_{\ul{u}} \geq \epsilon$.
By Theorem~\ref{gblq-th20} there is a unique
solution $u^{\varepsilon} \in C^{2,\alpha} (\bM)$ of
\eqref{cma-K700'} with $u^{\varepsilon} \geq \phi$ on $\bM$
for $\varepsilon \in (0, \epsilon]$.

By the estimates in Section~\ref{gblq-G} we have
 \begin{equation}
\label{cma-K730}
 |u^{\varepsilon}|_{C^1(\bM)} \leq C_1, \; \sup_{M} \Delta u^{\varepsilon} \leq
C_2 (1+ \sup_{\partial M} \Delta u^{\varepsilon}),
\; \mbox{independent of $\varepsilon$}.
\end{equation}
On the boundary $\partial M$, the estimates in Section~\ref{gblq-B} for
the pure tangential and mixed tangential-normal second derivatives are
independent of $\varepsilon$, i.e.,
\begin{equation}
\label{cma-K750}
 |u^{\varepsilon}_{\xi \eta}|, \;  |u^{\varepsilon}_{\xi \nu}| \leq C_3,
\; \forall \; \xi, \eta  \in T \partial M, |\xi|, |\eta| = 1, \;\;
\; \mbox{independent of $\varepsilon$}.
\end{equation}
where $\nu$ is the unit normal to $\partial M$.
For the estimate of the double normal derivative
$u^{\varepsilon}_{\nu \nu}$, note that
$\partial M = N \times \partial S$ and $T_{C} \partial M = T N$;
this is the only place we need the assumption
$M = N \times S$ so Theorem~\ref{gblq-th50} actually holds for local
product spaces.
Thus,
\begin{equation}
\label{cma-K740}
\chi_{\xi \bar{\xi}} + u^{\varepsilon}_{\xi \bar{\xi}}
   = \chi_{\xi \bar{\xi}} + \ul{u}_{\xi \bar{\xi}} \geq c_0 \;\;
\forall \; \xi \in T_{C} \partial M = T N, \; |\xi| = 1
\end{equation}
where $c_0$ depends only on $\ul{u}$.
Therefore, 
\begin{equation}
\label{cma-K760}
 |u^{\varepsilon}_{\nu \nu}| \leq C,
\;\; \mbox{independent of $\varepsilon$ on $\partial M$}.
\end{equation}

Finally, from $\sup_M |\Delta u^{\varepsilon}| \leq C$ we see that
$|u^{\varepsilon}|_{C^{1, \alpha} (\bM)}$ is bounded for any
$\alpha \in (0, 1)$. Taking a convergent subsequence we obtain a
solution $u \in C^{1, \alpha} (\bM)$
of \eqref{gblq-I10}. 
By Remark~\ref{gblq-remark-R10}, $u \in C^{1, 1} (\bM)$ when
$M$ has nonnegative bisectional curvature.
\end{proof}

\bigskip

\section{Proof of Theorem~\ref{gblq-TR}}
\label{gblq-T}
\setcounter{equation}{0}
\medskip

By a theorem of Harvey and Wells~\cite{HW72} (see also \cite{NW69})
there exists a strictly plurisubharmonic function $\rho \in C^3 (\bar{N})$,
where $N \subseteq M$ is a neighborhood of $X$, such that
$\rho^{-1} (\{0\}) = X$, $\rho = 1$ on $\partial N$ and
$\nabla \rho \neq 0$ on $\bar{N} \setminus X$.
Let $\ul{u} = a \rho$. We can fix $a > 2$ sufficiently large so that
$\chi_{\ul{u}} \geq 2 \omega$ in $\bar{N}$.
For $0 < \epsilon \leq 1$ denote
$M_{\epsilon} = \{\ul{u} < \epsilon\}$ and let
$\ul{u}^{\epsilon} \in C^3 (\bar{N})$ be a function such that
$\ul{u}^{\epsilon} = \ul{u}$ in $M_{\epsilon/2}$,
$\ul{u}^{\epsilon} \in C^{\infty} (\bar{M}_{a/2} \setminus M_{\epsilon})$
and $\ul{u}^{\epsilon} \rightarrow \ul{u}$ in $C^3 (\bar{N})$ as
$\epsilon \rightarrow 0$.
We denote $M_{\epsilon, \delta} = \{\ul{u}^{\epsilon} < \delta\}$.
Given $0 < \delta < 1$ 
we see that for all $\epsilon$ sufficiently small, 
$ M_{\epsilon/2} \subseteq M_{\epsilon, \delta}
  \subseteq  M_{\epsilon, 1} \subseteq M_{a/2}$ 
and
$\nabla \ul{u}^{\epsilon} \neq 0$,
   $\chi_{\ul{u}^{\epsilon}} \geq \omega$
on $\bar{M}_{\epsilon, 1} \setminus M_{\epsilon, \delta}$. 

We now consider the following Dirichlet problem
\begin{equation}
\label{MA-epsilon}
\left\{ \begin{aligned}
&\, (\chi_{u})^n = \delta \omega^n \;\;
   \mbox{in $\bar{M}_{\epsilon, 1} \setminus M_{\epsilon, \delta}$},\\
&\, u = \ul{u}^{\epsilon} \;\;
    \mbox{on $\partial (M_{\epsilon, 1} \setminus M_{\epsilon, \delta})$}.
\end{aligned} \right.
\end{equation}
Note that $\ul{u}^{\epsilon}$ is a subsolution of \eqref{MA-epsilon}
and $\partial (M_{\epsilon, 1} \setminus M_{\epsilon, \delta})$ is smooth.
By Theorem~\ref{gblq-th20} there exists a unique solution
$u^{\epsilon, \delta} \in \cH_{\chi} \cap
 C^{\infty} (\bar{M}_{\epsilon, 1} \setminus M_{\epsilon, \delta})$
to problem \eqref{MA-epsilon}.
It follows from the maximum principle that 
$\ul{u}^{\epsilon} \leq u^{\epsilon, \delta} \leq 1$ in
$\bar{M}_{\epsilon, 1} \setminus M_{\epsilon, \delta}$. 
By (the proof of) Proposition~\ref{gblq-prop-G10}, 
\begin{equation}
\label{gblq-T10}
 \max_{\bar{M}_{\epsilon, 1} \setminus M_{\epsilon, \delta}}
|\nabla u^{\epsilon, \delta}| \leq C \Big(1 + 
\max_{\partial ({M}_{\epsilon, 1} \setminus M_{\epsilon, \delta})}
|\nabla u^{\epsilon, \delta}|\Big) 
\end{equation}
where $C$ depends on $|\ul{u}^{\epsilon}|_{C^3}$.
Since $\ul{u}^{\epsilon} \rightarrow \ul{u}$ in
$C^3 (\bar{N})$ as $\epsilon \rightarrow 0$,  we see that $C$ can be chosen 
uniformly in $\epsilon$. 
Consequently, there exists a sequence $\epsilon_k \rightarrow 0$ such that 
$u^{\epsilon_k, \delta}$ converges to a function 
$u^{\delta} \in C^{0,1} (\bar{M}_{1} \setminus M_{\delta})$ 
as $k$ tends to infinity. Moreover, $u^{\delta}$ is an admissible weak 
solution (\cite{BT76}) of the problem
\begin{equation}
\label{MA-delta}
\left\{ \begin{aligned}
&\, (\chi_{u^{\delta}})^n  = \delta \omega^n \;\;
   \mbox{in $\bar{M}_{1} \setminus M_{\delta}$},\\
&\, u^{\delta}  = \ul{u} \;\;
    \mbox{on $\partial (M_{1} \setminus M_{\delta})$}
\end{aligned} \right.
\end{equation}
and
\begin{equation}
\label{gblq-T20}
 \max_{\bar{M}_{1} \setminus M_{\delta}}
|\nabla u^{\delta}| \leq C \Big(1 +
\max_{\partial ({M}_{1} \setminus M_{\delta})}
|\nabla u^{\delta}|\Big). 
\end{equation}
Obviously, $|\nabla u^{\delta}| \leq |\nabla \ul{u}| \leq C$ on 
$\partial M_{1}$ where $C$ is independent of $\delta$. We wish to show that 
\begin{equation}
\label{gblq-T30}
 |\nabla u^{\delta}| \leq C \;\;\mbox{on $\partial M_{\delta}$,  
independent of $\delta$}
\end{equation}
and therefore
\begin{equation}
\label{gblq-T20'}
 \max_{\bar{M}_{1} \setminus M_{\delta}}
|\nabla u^{\delta}| \leq C, \;\; \mbox{independent of $\delta$}.
\end{equation}

Consider an arbitrarily fixed point $p \in X$. Let $\nu \in T_p M$ 
be a unit normal vector to $X$, i.e, $\nu \in N_p X$.  Since $X$ is 
totally real and $\dim X = n$, we see that $J \nu \in T_p X$. 
Through $p$ there exists a complex curve $S = S (p, {\nu}) \subset M$ such 
that $T_p S$ is spanned by $\nu$ and $J \nu$. We may assume 
$S$ to be a geodesic disk about $p$ of radius 
$\gamma > 0$ which is independent of $p$ and $\nu \in N_p X$. 
Moreover,  since $X$ is totally real and $C^3$, we may assume 
(choosing $\gamma$ sufficiently small) that $X \cap B_{\gamma'}$
is a connected curve for any geodesic disk $B_{\gamma'} (p) $ about 
$p$ of radius $\gamma' \leq \gamma$. 
Let $\Gamma = X \cap S$. We see that $\Gamma$
divides $S$ into two components; let $S^+$ denote that one
to which $\nu$ is the interior unit normal
and $B_{\gamma'}^+ (p) = S^+ \cap B_{\gamma'} (p)$. 

For $\delta \geq 0$ sufficiently small, let $h^{\delta}$ be the solution 
of the problem
\begin{equation}
\label{gblq-T35}
\left\{
\begin{aligned}
& \Delta_S h + \tr (\chi|_S) = 0 \;\; \mbox{in $S^+_{\delta}$},\\
& h = \eta^{\delta} \;\; \mbox{on $\partial S^+_{\delta}$}
\end{aligned} \right.
\end{equation}
where $\chi|_S$ is the restriction of $\chi$ on $S$,
$S^+_{\delta} = S^+ \cap M_{\delta}$, 
 and $\eta^{\delta}$ is 
a smooth function on $\partial S^+_{\delta}$ with 
$\eta^{\delta} = \delta$ on 
$\partial S^+_{\delta} \cap B_{\gamma/2} (p)$ and 
$\eta^{\delta} = 1$ on 
$\partial S^+_{\delta} \setminus \partial M_{\delta}$. 
By the elliptic regularity theory $h^{\delta} \in C^{2, \alpha} 
(\ol{S^+_{\delta} \cap B_{\gamma'} (p)})$ for all 
$\gamma' < \gamma$, and 
\begin{equation}
\label{gblq-T40}
 |h^{\delta}|_{C^{2, \alpha} 
   (\ol{S^+_{\delta} \cap B_{\gamma'} (p)})} \leq C \;\;
\mbox{independent of $\delta$}. 
\end{equation}

Since $h^{\delta} \geq u^{\delta}$ on 
$\partial (S^+ \cap M_{\delta})$ and
$\Delta_S  u^{\delta} + \tr (\chi|_S) \geq 0$ in $S^+_{\delta}$,
by the maximum principle
$h^{\delta} \geq u^{\delta}$ in $S^+ \cap M_{\delta}$.
Consequently, 
\begin{equation}
\label{gblq-T50}
 \frac{\partial \ul{u}}{\partial \vec{n}}
  \leq \frac{\partial u^{\delta}}{\partial \vec{n}} 
  \leq  \frac{\partial h^{\delta}}{\partial \vec{n}}  \;\;
  \mbox{on $S^+ \cap \partial M_{\delta}$} 
\end{equation}
where $\vec{n}$ is the interior unit normal vector 
field to $S^+ \cap \partial M_{\delta}$ in $S^+_{\delta}$. 

Note that $\{S (p, {\nu}): \nu \in N_p X, |\nu|=1\}$ forms a
foliation of a neighborhood of $p$ which contains a geodesic ball
about $p$ of a fixed radius (independent of $p$) in $M$. 
Let $q \in \partial M_{\delta}$ and $\vec{n}$ be the unit normal vector
to $\partial M_{\delta}$ at $q$.
When $\delta$ is sufficiently small, there exists $p \in X$ and a unique 
$\nu \in N_p X \subset T_p M$, $|\nu| = 1$ such that 
$q \in S (p, {\nu})$ and $\vec{n} \in T_q S$ and therefore is conormal
to $\partial S^+_{\delta} (p, \nu)$ at $q$. Consequently, by \eqref{gblq-T40}
and \eqref{gblq-T50},
\begin{equation}
\label{gblq-T60}
|\nabla u^{\delta} (q)|
     = \Big|\frac{\partial u^{\delta}}{\partial \vec{n}} (q)\Big|
  \leq \max \Big\{\Big|\frac{\partial \ul{u}}{\partial \vec{n}} (q)\Big|, 
         \Big|\frac{\partial h^{\delta}}{\partial \vec{n}} (q)\Big|\Big\}  
  \leq C,
 \;\; \mbox{independent of $q$, $\delta$}. 
\end{equation}
This proves \eqref{gblq-T30} and \eqref{gblq-T20'}.

We observe that if $\delta' < \delta$ then $u^{\delta'} \geq u^{\delta}$
on $\partial (M_1 \setminus M_{\delta})$. By the maximum principle, 
$u^{\delta'} \geq u^{\delta}$
in $\bM_1 \setminus M_{\delta}$ if $\delta' < \delta$.
Therefore $u^{\delta}$ converges to a function $u$ as 
$\delta \rightarrow 0$ pointwise in $\bar{M}_1 \setminus X$. From 
\eqref{gblq-T20'} we see that $u \in C^{0,1} (\bar{M}_1 \setminus X)$ and 
solves \eqref{gblq-I10H} (in the weak sense of Bedford-Taylor~\cite{BT76}). 

Let $p \in X$ and $\nu \in N_p X$, $|\nu| = 1$, and $S = S (p, \nu)$ be as 
before. Let $h^{\delta}$ be the solution of problem \eqref{gblq-T35}
for $\delta \geq 0$. We have 
$\ul{u} \leq u^{\delta} \leq h^{\delta} \leq h^0$ in $S^+_{\delta}$
for all $\delta > 0$. Thus,
$\ul{u} \leq u \leq h^0$ in $S^+$. This shows that $u$ can be extended to 
$u \in C^{0,1} (\bar{M}_1)$ with $u = 0$ on $X$.

The proof of Theorem~\ref{gblq-TR} is complete.

\bigskip

\section{A Dirichlet problem related to Donaldson conjecture}
\label{gblq-S}
\setcounter{equation}{0}
\medskip

Let $(M^n, g)$ be a compact Hermitian manifold without boundary. The space of
Hermitian metrics
\begin{equation}
\label{cma-K220'}
 \cH = \{\phi \in C^{2} (M): \omega_{\phi} > 0\}
\end{equation}
is an open subset of $C^2 (M)$. The tangent space
$T_{\phi} \cH$ of $\cH$ at $\phi \in \cH$ is naturally identified to
$C^2 (M)$. Following \cite{Mabuchi87}, \cite{Semmes92} and
\cite{Donaldson99} one can define a Riemannian metric on $\cH$
using the $L^2$ inner product on $T_{\phi} \cH$
with respect to the volume form of $\omega_{\phi}$:
\begin{equation}
\label{cma-K620}
 \langle\xi, \eta\rangle_{\phi} = \int_M \xi \eta \, (\omega_{\phi})^n,
 \;\; \xi, \eta \in T_{\phi} \cH.
\end{equation}
Accordingly, the length of a regular curve
$\varphi: [0, 1] \rightarrow \cH$ is
 \begin{equation}
\label{cma-K630}
 L (\varphi) =
 \int_0^1 \langle\dot{\varphi},
   \dot{\varphi}\rangle_{\varphi}^{\frac{1}{2}} dt.
 \end{equation}
Henceforth $\dot{\varphi} = \partial \varphi/\partial t$
and $\ddot{\varphi} = \partial^2 \varphi/\partial t^2$.
When $\omega$ is K\"ahler, the geodesic equation takes the form
 \begin{equation}
\label{cma-K640}
 \ddot{\varphi} - |\nabla \dot{\varphi}|_{\varphi}^2 = 0,
 \end{equation}
 or in local coordinates
 \begin{equation}
\label{cma-K640'}
 \ddot{\varphi} - g (\varphi)^{j\bk} \dot{\varphi}_{z_j} \dot{\varphi}_{\bz_k}
 = 0.
 \end{equation}
Here $\{g (\varphi)^{j\bk}\}$ is the inverse matrix of
$\{g (\varphi)_{j\bk}\} = \{g_{j\bk} + \varphi_{j\bk}\}$.
It was observed by Donaldson~\cite{Donaldson99}, Mabuchi~\cite{Mabuchi87}
and Semmes~\cite{Semmes92} that equation~\eqref{cma-K640}
reduces to a homogeneous complex Monge-Amp\`ere equation in $M \times A$
where $A = [0,1] \times \bfS^1$.
Let
\[ w = z_{n+1} = t + \sqrt{-1} s \]
 be a local coordinate of $A$.
If we view a smooth curve in $\cH$ as a function on
$M \times [0, 1]$ and therefore a rotation-invariant function
(constant in $s$) on $M \times A$, then
a geodesic $\varphi$ in $\cH$ satisfies
 \begin{equation}
\label{cma-K670}
 (\tilde{\omega}_{\varphi})^{n+1} \equiv \Big(\tilde{\omega} + \frac{\sqrt{-1}}{2} \partial
\bar{\partial} \varphi\Big)^{n+1} = 0 \;\;\; \mbox{in $M \times A$}
\end{equation}
where
 \begin{equation}
\label{cma-K680}
 \tilde{\omega} = \omega + \frac{\sqrt{-1}}{2} \partial \bar{\partial} |w|^2
  = \frac{\sqrt{-1}}{2} \Big(\sum_{j, k \leq n}  g_{j\bk} dz_j \wedge
  d\bz_k + dw \wedge d\bw\Big)
\end{equation}
is the lift of $\omega$ to $M \times A$.
Conversely, if $\varphi \in C^2 (M \times A)$ is a
rotation-invariant solution of \eqref{cma-K670} such that
\begin{equation}
\label{cma-K690}
 \varphi (\cdot, w) \in \cH, \;\; \forall \; w \in A,
 \end{equation}
then $\varphi$ is a geodesic in $\cH$.

In the K\"ahler case, Donaldson~\cite{Donaldson99} conjectured that
$\cH^{\infty} \equiv \cH \cap C^{\infty} (M)$ is geodesically convex,
i.e., any two functions in $\cH^{\infty}$ can be connected by a smooth
geodesic.
More precisely,

\begin{conjecture}[Donaldson~\cite{Donaldson99}]
\label{Donaldson-C99o}
Let $M$ be a compact K\"ahler manifold without boundary and
$\rho \in C^{\infty} (M \times \partial A)$ such that
$\rho (\cdot, w) \in \cH$ for $w \in \partial A$. Then there exists
a unique solution $\varphi$  of the Monge-Amp\`ere equation
\eqref{cma-K670} satisfying \eqref{cma-K690} and the boundary
condition $\varphi = \rho$.
\end{conjecture}

The uniqueness was proved by Donaldson~\cite{Donaldson99} as a
consequence of the maximum principle. In \cite{Chen00}, X.-X. Chen
obtained the existence of a weak solution
with $\Delta \varphi \in L^{\infty} (M \times A)$; see also the recent
work of Blocki~\cite{Blocki} who proved that the solution is in
$C^{1,1} (M \times A)$ when $M$ has nonnegative bisectional curvature.
As a corollary of Theorem~\ref{gblq-th50} these results can be extended
to the Hermitian case.

\begin{theorem}
\label{cma-K-thm30}
Let $M$ be a compact Hermitian manifold without boundary.
and let
$\varphi_0, \varphi_1 \in \cH \cap C^4 (M)$.
 There exists a unique (weak) solution $\varphi \in C^{1,\alpha} (M \times
 A)$, $\forall \; 0 < \alpha < 1$, with $\tilde{\omega}_{\varphi} \geq 0$
 and $\Delta \varphi \in L^{\infty} (M \times A)$
 of the Dirichlet problem
  \begin{equation}
\label{cma-K700}
 \left\{ \begin{aligned}
 & (\tilde{\omega}_{\varphi})^{n+1}  = 0 \;\; \mbox{in $M \times A$} \\
 & \varphi = \varphi_0 \;\; \mbox{on $M \times \Gamma_0$}, \\
 &  \varphi = \varphi_1 \;\; \mbox{on $M \times \Gamma_1$}
  \end{aligned} \right.
\end{equation}
where $\Gamma_0 = \partial A|_{t=0}$, $\Gamma_1 = \partial A|_{t=1}$.
Moreover, $\varphi \in C^{1,1} (M \times A)$ if $M$ has nonnegative
bisectional curvature.
\end{theorem}

\begin{proof}
In order to apply Theorem~\ref{gblq-th50} to the Dirichlet
problem~\eqref{cma-K700} we only need to construct a strict subsolution.
This is easily done for the
annulus $A = [0,1] \times \bfS^1$. Let
\[ \ul{\varphi} = (1 - t) \varphi_0 + t \varphi_1 + K (t^2 - t). \]
Since $\varphi_0, \varphi_1 \in \cH (\omega)$ we see that
$\tilde{\omega}_{\ul{\varphi}} > 0$ and
$(\tilde{\omega}_{\ul{\varphi}})^{n+1} \geq 1$
for $K$ sufficiently large.
\end{proof}

\begin{remark}
By the uniqueness $\varphi$ is rotation invariant (i.e., independent of $s$).
\end{remark}

\end{document}